\documentclass[sn-mathphys]{sn-jnl}

\usepackage{amssymb,amsmath,amssymb,mathrsfs}

\jyear{2022}%

\theoremstyle{thmstyleone}%
\newtheorem{theorem}{Theorem}[section]
\newtheorem{lemma}{Lemma}[section]
\newtheorem{proposition}{Proposition}[section]

\newtheorem{remark}{Remark}[section]%

\newtheorem{definition}{Definition}[section]%

\newtheorem{assumption}{Assumption}[section]
\newtheorem{problem}{Problem}

\begin{document}

\title[Stability for the Helmholtz equation in periodic structures]{Stability for the Helmholtz equation in deterministic and random periodic structures}


\author[1]{\fnm{Gang} \sur{Bao}}

\author[2]{\fnm{Yiwen} \sur{Lin}}

\author[1]{\fnm{Xiang} \sur{Xu}}

\affil[1]{\orgdiv{School of Mathematical Sciences}, \orgname{Zhejiang University}, \orgaddress{\city{Hangzhou}, \postcode{310027}, \country{China}}}

\affil[2]{\orgdiv{School of Mathematical Sciences and Institute of Natural Sciences}, \orgname{Shanghai Jiao Tong University}, \orgaddress{\city{Shanghai}, \postcode{200240}, \country{China}}}


\abstract{Stability results for the Helmholtz equations in both deterministic and random periodic structures are proved in this paper.
	Under the assumption of excluding resonances, by a variational method and Fourier analysis in the energy space,
	the stability estimate for the Helmholtz equation in a deterministic periodic structure is established.
	For the stochastic case, by introducing a variable transform, the variational formulation of the scattering problem in a random domain is reduced to that in a definite domain with random medium. Combining the stability result for the deteministic case with regularity and stochastic regularity of the scattering surface, Pettis measurability theorem and Bochner's Theorem further yield the stability result for the scattering problem by random periodic structures. Both stability estimates are explicit with respect to the wavenumber.}

\keywords{Helmholtz equation, periodic structure, high wavenumber, stability, estimate explicit on wavenumber}



\maketitle

\section{Introduction}
\label{S:1}

This paper is concerned with scattering of time-harmonic electromagnetic plane waves by periodic structures, which is known as gratings in optics.
The goal is to investigate stability properties of the Helmholtz equation in both deterministic and random periodic structures.

For the scattering by a periodic structure, considerable progress has been made mathematically in literatures.
Bao, Dobson and Cox \cite{Bao1995} reduced the scattering problem into a bounded domain problem by introducing a transparent boundary condition and proved that there exists a unique solution at all but a sequence of countable frequencies.
Lord and Mulholland \cite{Lord2013} used the variational formula  to derive a priori estimate  for periodic structures, which can be viewed as an extension of  \cite{Bao1996}.
More results on the Helmholtz and Maxwell equations in periodic structures can be found in \cite{Petit1980, Bao2001, Bao2021book}.

Recently, there is an increasing interest in the study of wavenumber-explicit bounds for scattering problems.
Chandler-Wilde and Monk \cite{Chandler2005} obtained a priori bounds explicit with the wavenumber for the scattering problem by unbounded rough surfaces.
Hetmaniuk \cite{Hetmaniuk2007} established stability estimates for the Helmholtz equation with mixed boundary conditions. Esterhazy and Melenk \cite{Esterhazy2012} further established an estimate for bounded Lipschitz domains with a Robin boundary condition.
The stability for the scattering problem by a large rectangular cavity was obtained by Bao, Yun and Zhou \cite{Bao2012} in transverse electric polarization and was improved and extended to transverse magnetic polarization in  \cite{Bao2016}.
Du, Li and Sun \cite{LiBuyang2015} presented a numerical study of the stability estimate of the scattering by a rectangular cavity.
Wavenumber-explicit stability on the scattering problem by an obstacle in homogeneous media case \cite{Spence2014} and heterogeneous media case \cite{Pembery2019} were obtained
under the assumption that the scatterer $D$ is a star-shaped domain satisfying nontrapping conditions.
However, for periodic structures, little is known about stability analysis with explicit dependence on the wavenumber, which will be derived in this paper.
As noted in \cite{Bao2012, Bao2016, Chandler2005}, both the geometry and the type of boundary condition strongly affect the wavenumber-dependent stability.
The transparent boundary conditions (TBC) of periodic structures  are different from those of open cavities and unbounded rough surfaces, which leads to additional difficulties on resonances.
To overcome the difficulties, a uniform distance $\varepsilon$ is introduced in this paper and a $k$-explicit stability for periodic structures is established by a combination of the variational method and Fourier analysis.

We also analyze the stability for scattering by random periodic structures in this paper.
Progress has been made recently in the exterior scattering problem by an obstacle with randomness\cite{Spence2014, Hiptmair2018} and scattering problems in random media \cite{Pembery2019, Pembery2020}.
However, no stability result is available for scattering by random periodic structures, which has many applications in diffractive optics \cite{Rico-Garcia2009}.
Numerical methods for the scattering by random periodic structures have been developed recently in \cite{Feng2018, BaoLinSINUM2020}.
In this work, we focus on the stability for gratings with uncertainty.
One difficulty is the lack of compactness. In fact, since both the deterministic and the stochastic Helmholtz equations are not coercive and the necessary compactness results are not valid any more in Bochner spaces, Fredholm theory cannot be used to the stochastic Helmholtz equation to compensate for the lack of coercivity as for the deterministic Helmholtz equation. Consequently, it is difficult to obtain the well-posedness for the random case directly as that for the deterministics case.
In this work, we employ Pettis measurability theorem and Bochner's theorem as in \cite{Pembery2020} to obtain the stability estimate explicit on the wavenumber for the random case based on the stability result for the deterministic case.
However, the randomness of the integral domain for the scattering by random periodic structures prevents a direct application of the general framework in  \cite{Pembery2020}, which leads to the other difficulty.
To overcome this difficulty, a variable transform is introduced here so that the transformed random variational form is defined on a deterministic domain with random coefficients. Similar transformation idea has also been used for other types of random differential equations \cite{Xiu2006}.
Therefore, for the stochastic case, by integrating the deterministic result over the probability space and introducing a transform to change the stochastic integral area into a definite one,  regularity and stochastic regularity of the scattering surface, Pettis measurability theorem and Bochner's theorem yield the well-posedness and a stability estimate of our model problem.

The rest of the paper is as follows. In Section \ref{sec:results}, the model problem is introduced, the variational formulation is described, and the results (i.e., Theorems \ref{thm:1}-\ref{thm:3}) are presented.
Theorem \ref{thm:1} concerns the stability for the Helmholtz equation in a deterministic periodic structure.
Theorem \ref{thm:3} gives the stability estimate for large wavenumber for the Helmholtz equation in a random periodic structure.
The following sections are devoted to the proofs of the two results.
Section \ref{sec:thm1} and Section \ref{sec:thm23} give the detailed proofs of Theorem \ref{thm:1} and Theorem \ref{thm:3} respectively, followed by
conclusion given in Section \ref{sec:conslusion}.

\section{Main results}
\label{sec:results}

\subsection{Model problem}
Consider a plane wave incident on a random periodic structure
$$S := \{x\in \mathbb{R}^2: x_2=f(\omega;x_1),\omega \in \Omega,  x_1 \in [0,\Lambda]\},$$
which is characterized by the wavenumber $k$ ruled on a perfect conductor.
Here, $\omega \in \Omega$ denotes the random sample in a complete probability space $(\Omega, \mathcal{F}, \mu)$,
$x=(x_1,x_2)\in \mathbb{R}^2$ are the spatial variables, and the random surface $f: \Omega \times \mathcal{X}   \rightarrow \mathbb{R}$ is a stationary Gaussian process, each of which is a Lipschitz function, i.e. $f\in L^2(\Omega;Lip)$, where $Lip$ is the space of all Lipschitz functions.
The medium and material are assumed to be invariant in the $x_3$ direction and  $\Lambda$-periodic in the $x_1$ direction.
There are two fundamental polarizations for the electromagnetic fields: transverse-electric (TE) and transverse-magnetic (TM) polarization. Here, we consider the random periodic perfectly conducting grating problem for TE polarization.

\begin{figure}[h]
	\centering
	\includegraphics[width=0.9\textwidth]{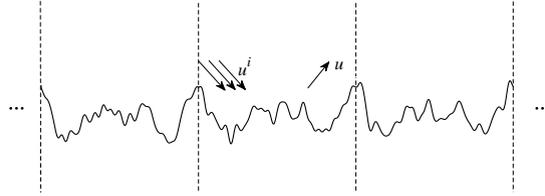}\\
	\caption{Problem geometry}\label{fig:geometry}
\end{figure}

As shown in Figure \ref{fig:geometry}, the grating is illuminated from above by a time-harmonic plane wave
$	u^{i}=e^{i\alpha x_1-i\beta x_2},$
where
$ \alpha = k \sin\theta,~\beta=k \cos\theta$,
$\theta\in(-\pi/2,\pi/2)$ is the incident angle with respect to the positive $x_2$-axis.
Denote $$D^+= \{x\in\mathbb{R}^2: x_2>f(\omega; x_1),\ \omega \in \Omega,  x_1 \in [0,\Lambda]\}.$$ Since the medium below $S$ is perfectly electric conducting, the scattering problem for TE polarization can be modeled by the following two-dimensional Helmholtz equation with the homogeneous boundary condition:
\begin{eqnarray}\label{eq:u0}
	\left\{\begin{array}{ll}
		\Delta u(\omega ; \cdot)+k^{2} u(\omega ; \cdot)=0 & \text { in } \Omega \times D^+, \\
		u(\omega;\cdot)=0 & \text { on } \Omega \times S ,
	\end{array}\right.
\end{eqnarray}
where $u(\omega ; \cdot)$ is the total field.

\subsection{Variational form}
\label{sec:variational}
In order to obtain a stability estimate, an important step is to reduce the infinite scattering problem into a bounded domain problem by introducing a transparent boundary condition \cite{Bao1995SINUM}.

Since the total field can be decomposed into
$
u(\omega;\cdot) = u^{i} + u^s(\omega;\cdot)
$
where
the incident wave satisfies
\begin{equation*}
	\Delta u^{i}+k^{2} u^{i}=0 \quad  \text { in } \Omega \times D^+,
\end{equation*}
the scattered field $u^s(\omega;\cdot)$ satisfies
\begin{eqnarray}\label{eq:u}
	\left\{\begin{array}{ll}
		\Delta u^s(\omega ; \cdot)+k^{2} u^s(\omega ; \cdot)=0 & \text { in } \Omega \times D^+, \\
		u^s(\omega ; \cdot)+u^{i}=0 & \text { on } \Omega \times S .
	\end{array}\right.
\end{eqnarray}

Denote by $\Gamma = \{ (x_1,x_2)\in\mathbb{R} ^2 : x_2=b, x_1 \in [0,\Lambda] \}  $
with
$b > \max_{\omega\in \Omega, x_1\in(0,\Lambda)}  f(\omega;x_1). $
The scattered field $u^s$ above $\Gamma$ admits the Rayleigh expansion
\begin{equation}\label{Rayleigh}
	u^s(\omega;x_1,x_2)=\sum_{n \in Z}  A_{n}  e^{i\left(\alpha_{n}  x_1+\beta_{n}  x_2\right)} ,
\end{equation}
where
\begin{equation}\label{alpha_n}
	\alpha_{n}=\alpha+\dfrac{2 \pi n}{\Lambda}, \beta_{n}^{2}=k^{2}-\alpha_{n}^{2} \text { with } \Im \beta_{n} \geq 0.
\end{equation}

\begin{assumption}\label{assump:k}
	(Exclusion of resonances). Assume there exists $\varepsilon>0$, such that $\vert k-\alpha_n \vert\geq \varepsilon, \ \forall n \in \mathbb{Z} $.
\end{assumption}

\begin{remark}
	A constant  $\varepsilon$ is introduced here to exclude	all possible resonances ($\kappa = \alpha_n$) uniformly and thus a stability explicit with the wavenumber $k$ is obtained by a combination of the variational method and Fourier analysis.
\end{remark}

Define a boundary operator $T:H^{\frac{1}{2}}(\Gamma) \rightarrow H^{-\frac{1}{2}}(\Gamma)$ by
\begin{equation*}
	(T v)\left(x_{1} \right)=\sum_{n \in \mathbb{Z} }  \mathrm{i}  \beta_{n}  v_{n}  e^{\mathrm{i}  \alpha_{n}  x_{1} } ,
\end{equation*}
where
$v_{n} (\omega)=\frac{1} {\Lambda}  \int_{0} ^{\Lambda}  v\left(\omega;x_{1} ,b\right) e^{-\mathrm{i}  \alpha_{n}  x_{1} }  \mathrm{d} x_{1} , $
which yields a transparent boundary condition on $\Gamma$:
\begin{equation*}
	\partial_{\nu} u = T u +g(x_1) \quad \text{ on } \Gamma,
\end{equation*}
and $g(x_1)= \partial_{\nu} u^{i} - T u^{i} = -2 i \beta e^{i \alpha x_1 - i \beta b}$.

Denote by $D= \{x\in\mathbb{R}^2: f(\omega; x_1)<x_2<b,\ \omega \in \Omega,  x_1 \in [0,\Lambda]\}$. The scattering problem \eqref{eq:u0} in an unbounded domain may be reduced to the following problem in the bounded domain:
\begin{equation}\label{eq:u_total}
	\left\{\begin{array}{ll}
		\Delta u(\omega;\cdot)+k^{2} u(\omega;\cdot)=0 & \text { in } \Omega \times D, \\
		u(\omega;\cdot)=0 & \text { on } \Omega \times S, \\
		\partial_{\nu}u(\omega;\cdot) = Tu(\omega;\cdot)+g(x) & \text{ on } \Omega\times \Gamma.
	\end{array}\right.
\end{equation}

Motivated by uniqueness, we seek for quasi-periodic solutions of $u(\omega;\cdot)$, that is, $u(\omega;x_1+\Lambda,x_2) = e^{i \alpha \Lambda} u(\omega;x_1,x_2)$.
Let $\tau$ denote a trace operator from $D$ or $S$ to $\Gamma$.
Define the function spaces inside $D$ by
\begin{equation*}
	\begin{aligned}
		H^1(D)& = \{ v\in L^2(D): \nabla v \in L^2(D)^2 \},\\
		H^1_S(D) &= \{v\in H^1(D): \tau v=0 \text{ on } S\}, \\
		H^1_{S,qp}(D)&=\{v\in H^1_S(D): v(0,y) = v(\Lambda,y)e^{-i\alpha \Lambda}\}.
	\end{aligned}
\end{equation*}

Let $c:\Omega \rightarrow \mathcal{C}$ be defined by
\begin{equation}\label{def:c}
	c(\omega)=f(\omega),
\end{equation}
where $\mathcal{C}:= Lip$ is the space of all Lipschitz functions.

\textit{Variational form.}
Define the sesquilinear forms $\mathcal{A}$ on $L^2(\Omega;H^1_{S,qp}(D)) \times L^2(\Omega;H^1_{S,qp}(D))$ and  $\tilde{a}_{c(\omega)}$   on $H^1_{S,qp}(D) \times H^1_{S,qp}(D)$ by
\begin{equation}\label{def:AA}
	\mathcal{A}(u,v) := \int_{\Omega}\tilde{a}_{c(\omega)} (u,v) \mathrm{d} \mathbb{P}(\omega),
\end{equation}
and
\begin{equation}\label{aw}
	\tilde{a}_{c(\omega)} (u,v):= \int_D \nabla u \cdot \nabla \bar{v} - k^2 \int_D u \bar{v} - \int_\Gamma T \tau u \overline{ \tau v}.
\end{equation}
Define the antilinear functionals $\mathcal{G}$ on $L^2(\Omega;H^1_{S,qp}(D))$ and  $\tilde{G}_{c(\omega)}$  on  $H^1_{S,qp}(D)$ by
\begin{equation}\label{def:GG}
	\mathcal{G}(v) := \int_{\Omega} \tilde{G}_{c(\omega)}(v) \mathrm{d} \mathbb{P}(\omega),
\end{equation}
and
\begin{equation}\label{Gw}
	\tilde{G}_{c(\omega)}(v):= \int_\Gamma g \overline{ \tau v}.
\end{equation}
Consider the following variational form of the Helmholtz equation in a random periodic structure:
\begin{problem}\label{P3}
	Find $u\in L^2(\Omega;H^1_{S,qp}(D))$ such that
	\begin{equation}\label{variational}
		\mathcal{A} (u,v) = \mathcal{G} (v), \  \forall   v \in L^2(\Omega;H^1_{S,qp} (D)).
	\end{equation}
\end{problem}

\subsection{Stability results}

We first present a stability result explicit with respect to the wavenumber $k$ for the scattering problem by a deterministic periodic structure is given in Theorem \ref{thm:1}.

\begin{theorem}\label{thm:1}
	(Stability for the deterministic case). 	Fix one sample $\omega_0$. Let $\tilde{u}=u(\omega_0) \in H^1_{S,qp}(D)$ be the solution of the deterministic scattering problem. Under Assumption \ref{assump:k}, there exists a constant $C$ independent of  $b>1$ and large $k$, such that
	\begin{equation}\label{Destimate}
		\|\nabla \tilde{u} \|_{L^2(D)} + k \|\tilde{u}\|_{L^2(D)} \leq C \max \left\{ \dfrac{ b^2 {k} ^{2} } {\sqrt{\varepsilon} } ,b^3 {k} ^{\frac{5} {2} }   \right\}  \|g\|_{L^2(\Gamma)}.
	\end{equation}
\end{theorem}

\begin{remark} \label{re:1}
	For the estimate, $\varepsilon$ is independent of the wavenumber $k$ and the measured height $b$. Note that for small 	$\varepsilon$, wavenumber is close to resonances and then the corresponding bound is very large.
	If $\varepsilon\geq \frac{1} { b^2 k } $, then the estimate becomes
	\begin{equation}\label{reduced}
		\|\nabla  \tilde{u}  \|_{L^2(D)}  + k \| \tilde{u} \|_{L^2(D)}  \leq C b^3 k ^{\frac{5} {2} }     \|g\|_{L^2(\Gamma)}.
	\end{equation}
\end{remark}

Next, the first stability result in a random periodic structure is shown in Theorem \ref{thm:3} under Assumption \ref{assump:k} and Assumption \ref{cond:f}.

\begin{assumption}\label{cond:f}
	(Regularity and stochastic regularity of $f$). Assume the random structure  $ f\in L^2(\Omega;Lip)$.
\end{assumption}

\begin{remark}\label{remark:f0}
	(Random surfaces satisfying Assumption \ref{cond:f}). Motivated by the uncertainty quantification such as Karhunen-Lo\`eve expansion and other similar expansions,  it makes sense to consider $f$ as series expansions around the known deterministic surface $f_0$ satisfying the Lipschitz condition.
\end{remark}

\begin{theorem}\label{thm:3}
	(Stability for the random case).
	Under Assumption \ref{assump:k} and Assumption \ref{cond:f},
	there exists a unique solution $u\in L^2(\Omega; H_{S,qp}^1(D))$ to Problem \ref{P3}.
	Moreover, there exists a constant $C$ independent of  $b>1$ and large $k$,  such that
	\begin{equation}\label{Sestimate}
		\|\nabla u \|_{L^2(\Omega;L^2(D))} + k \|u\|_{L^2(\Omega;L^2(D))} \leq C \max \left\{ \dfrac{b^2 {k} ^{2} } {\sqrt{\varepsilon} }
		,b^3 {k} ^{\frac{5} {2} }   \right\}   \|g\|_{L^2(\Omega;L^2(\Gamma))}.
	\end{equation}
\end{theorem}

\section{Stability for the deterministic case}
\label{sec:thm1}

This section is devoted to the proof of Theorem \ref{thm:1}. Our approach based on the variational method and Fourier analysis consists of two steps: the first step is to estimate the norms of $u$ in the domain $D$ by the norms of $u$ on the boundary $\Gamma$ and the second step is to estimate the norms of $u$ on $\Gamma$ by the norm of $g$.

Although our approach is related to \cite{Bao2012,Bao2016} for the scattering by a large rectangular cavity, additional difficulties arise: (1) Because of the different model geometry, quasiperiodic solutions have to be considered for our model here. (2) Instead of the Dirichlet or Neumann  boundary condition for the cavity problem, the (nonlocal) transparent boundary conditions must be dealt with. (3) Unlike the cavity problem, our model problem may not have unique solutions due to the resonances. Therefore, new techniques must be developed to prove the stability result.

Before proceeding to the detailed proof,  some useful notations and preliminary results are introduced first.

\subsection{Notations and preliminary results}
\label{S:preliminary}

Let $\tilde{u}(x_1,x_2)=u(\omega_0;x_1,x_2)$ be the solution to the deterministic problem for a given sample $\omega_0$. Then the total field $\tilde{u}$ satisfies
\begin{equation}\label{eq:u_deter}
	\left\{\begin{array}{ll}
		\Delta \tilde{u}+k^{2} \tilde{u}=0 & \text { in } D, \\
		\tilde{u}=0 & \text { on } S, \\
		\partial_{\nu}\tilde{u} = T\tilde{u}+g & \text{ on }  \Gamma.
	\end{array}\right.
\end{equation}

Let $C$ denote positive constants which are independent of large $k$ and $b>1$. Denote by $ \alpha \simeq \beta$ if there exists a positive constant $C$ independent of large $k$ and $b>1$ such that $\frac{1}{C} \alpha \leq \beta\leq C \alpha$.
Without loss of generality, assume the period $\Lambda = 2\pi$ and denote $\tilde{u}$ as $u$ for simplicity.

The scattered field $u$ can be written in the  single-layer potential representation:
\begin{equation}\label{u1}
	u(x_1,x_2)=\int_{0}^{\Lambda} \phi(s) G(x_1, x_2 ; s, 0) \mathrm{d} s,
\end{equation}
where
$\phi \in L^2(0,\Lambda)$ is an unknown periodic density function and $G$ is a quasi-periodic Green function given explicitly as
$$
G(x_1, x_2 ; s, t)=\frac{i}{2 \pi} \sum_{n \in \mathbb{Z}} \frac{1}{\beta_{n}} e^{i \alpha_{n}(x_1-s)+i \beta_{n}\vert x_2-t\vert }, \quad(x_1,x_2) \neq(s, t).
$$
Without loss of generality, it is assumed that the normal incidence is with $2 \pi$-periodicity, i.e.
$\theta=0, \alpha=0, \beta=k, \Lambda = 2 \pi.$
Therefore
$\alpha_n=n$ and $\beta_n=\sqrt{k^2-n^2}$,
and thus
\begin{equation}\label{Gn}
	G(x_1,x_2; s, t)=\frac{i}{2 \pi} \sum_{n \in \mathbb{Z}} \frac{1}{\sqrt{k^2-n^2}} e^{i n(x-s)+i \sqrt{k^2-n^2}\vert y-t\vert }
	\text{~~and~~}
	\phi(s)=\sum_{n \in \mathbb{Z}} \phi_{n} e^{i n s}.
\end{equation}
Substituting \eqref{Gn} into \eqref{u1} yields
$$
u^s (x_1,x_2)=\sum_{n \in \mathbb{Z}} \psi_n e^{i nx_1+i \sqrt{k^2-n^2}  x_2}.
$$
Since $u=u^{inc} +u^s$, the solution $u$ can be expressed as
\begin{equation*}
	\begin{aligned}
		u(x_1,x_2)=&
		\displaystyle\sum_{n=0} ^{\infty}  a_{n}  \sin \left(\sqrt{{k} ^{2} -n^{2} }  x_2\right) \sin (n x_1) + \sum_{n=0} ^{\infty}  b_{n}  \cos \left(\sqrt{{k} ^{2} -n^{2} }  x_2\right) \cos (n x_1)\\
		& + i \left[ \displaystyle\sum_{n=0} ^{\infty}  c_n \cos \left(\sqrt{{k} ^{2} -n^{2} }  x_2\right) \sin (n x_1) + \sum_{n=0} ^{\infty}  d_{n}  \sin \left(\sqrt{{k} ^{2} -n^{2} }  x_2\right) \cos (n x_1)\right]\\
		\triangleq& \mu+i \nu.
	\end{aligned}
\end{equation*}
where $a_n=-\psi_n+\psi_{-n} $, $b_n=\psi_n+\psi_{-n} $, $n\in \mathbb{N} _+$, $a_0=-\psi_0+1$, $b_0=\psi_0+1$;  $c_n=-a_n$, $d_n=b_n$, $n\in \mathbb{N} _+$, $c_0=\psi_0-1$, $d_0=\psi_0+1$.
For the case of non-normal cases, we only need to replace $n$ with $\alpha_n$ and $\sqrt{k^2-n^2} $ with $\sqrt{k^2-{\alpha_n} ^2} $.
Since $\mu$ and $\nu$ can be studied similarly, only detailed study on $\mu$ is provided here. Also, for convenience, in the following proof, the notation $\mu$ is denoted as $u$.

Let $u_n$, $(\partial_{\nu}u)_n$, $(Tu)_n$ be the sine coefficients, and $v_n$, $(\partial_{\nu}v)_n$, $(Tv)_n$ be the cosine coefficients of the expansions of $u$, $\partial_{\nu}u$, $Tu$, on $\Gamma$, respectively. That is,
\begin{equation}\label{unvn}
	u_{n}:=
	a_{n} \sin \left(\sqrt{k^{2}-n^{2}} b\right),
	v_{n}:=
	b_{n} \cos \left(\sqrt{k^{2}-n^{2}} b\right),
\end{equation}
\begin{equation}\label{partialun}
	\left(\partial_{\nu} u\right)_{n}:=
	a_{n} \sqrt{k^{2}-n^{2}} \cos \left(\sqrt{k^{2}-n^{2}} b\right),
\end{equation}
\begin{equation}\label{partialvn}
	\left(\partial_{\nu} v\right)_{n}:=
	-b_{n} \sqrt{k^{2}-n^{2}} \sin \left(\sqrt{k^{2}-n^{2}} b\right).
\end{equation}
Hence
\begin{equation*}
	u\vert_{\Gamma}(x_1) = \sum_{n=0}^{\infty} u_n \sin(n x_1) + \sum_{n=0}^{\infty} v_n \cos(n x_1),
\end{equation*}
\begin{equation*}
	\partial_{\nu} u\vert_{\Gamma}(x_1) = \sum_{n=0}^{\infty} (\partial_{\nu} u)_n \sin(n x_1) + \sum_{n=0}^{\infty} (\partial_{\nu} v)_n \cos(n x_1),
\end{equation*}
and
\begin{equation*}
	Tu\vert_{\Gamma}(x_1) = \sum_{n=0}^{\infty} (Tu)_n \sin(n x_1) + \sum_{n=0}^{\infty} (Tv)_n \cos(n x_1).
\end{equation*}

\begin{definition}\label{def:AB}
	For $u =\displaystyle\sum_{n=0}^{\infty} u_n \sin(n x_1) + \sum_{n=0}^{\infty} v_n \cos(n x_1)$, define
	$$\|u\|_{A}=\sqrt{\sum_{n=0}^{\infty}\left( \lvert u_{n}\rvert^{2} +\lvert v_{n}\rvert^{2} \right) \frac{1}{\sqrt{\lvert k^{2}-n^{2}\rvert}+\frac{1}{b}}},$$
	and
	$$\|u\|_{B}=\sqrt{\sum_{n=0}^{\infty}\left( \lvert u_{n}\rvert^{2} +\lvert v_{n}\rvert^{2}\right) \left(\sqrt{\lvert k^{2}-n^{2}\rvert}+ \frac{1}{b}\right)}.$$
Here, the norms $A$ and $B$ are formerly dual to each other and behave like $H^{-1/2}$, $H^{1/2}$, respectively. Moreover, a simple calculation gives that
	\begin{equation}\label{norm}
		\int_{\Gamma} u_1(x_1) \overline{u_2(x_1)} \mathrm{d} x_1 \leq \pi \|u_1\|_{A}\|u_2\|_{B}.
	\end{equation}
\end{definition}

\begin{definition}
	Define the lower frequency set $\mathbf{L}$ and the higher frequency set $\mathbf{H}$ by
	$$\mathbf{L}=\left\{n \in \mathbb{N} \mid n < k \right\}
	\text{~and~~}
	\mathbf{H}=\left\{n \in \mathbb{N} \mid n > k \right\}.$$
\end{definition}

Then, $u$, $\partial_{\nu}u$, $T(u)$ on $\Gamma$ defined above can be separated into the lower frequency part and the higher frequency part:
$$
u(x_1) =\mathcal{V}_\mathbf{L} +\mathcal{V}_\mathbf{H}, \quad
\partial_{\nu} u(x_1) =\mathcal{P}_\mathbf{L} +\mathcal{P}_\mathbf{H},\quad
(T (u))(x_1) = \mathcal{T}_\mathbf{L} +\mathcal{T}_\mathbf{H},$$
with
\begin{equation*}
	\begin{aligned}
		&\mathcal{V}_\mathbf{L} =\sum_{n \in \mathbf{L}} u_{n} \sin (n x_1)
		+\sum_{n \in \mathbf{L}} v_{n} \cos (n x_1),\\
		&\mathcal{V}_\mathbf{H}=\sum_{n \in \mathbf{H}} u_{n} \sin (n x_1)
		+\sum_{n \in \mathbf{H}} v_{n} \cos (n x_1), \\
		&\mathcal{P}_\mathbf{L}=\sum_{n \in \mathbf{L}}\left(\partial_{\nu} u\right)_{n} \sin (n x_1)
		+\sum_{n \in \mathbf{L}}\left(\partial_{\nu} v\right)_{n} \cos (n x_1),\\
		&\mathcal{P}_\mathbf{H}=\sum_{n \in \mathbf{H}}\left(\partial_{\nu} u\right)_{n} \sin (n x_1)
		+\sum_{n \in \mathbf{H}}\left(\partial_{\nu} v\right)_{n} \cos (n x_1),\\
		&\mathcal{T}_\mathbf{L} =\sum_{n \in \mathbf{L}}(T u)_{n} \sin (n x_1)
		+\sum_{n \in \mathbf{L}}(T v)_{n} \cos (n x_1),\\
		&\mathcal{T}_\mathbf{H}=\sum_{n \in \mathbf{H}}(T u)_{n} \sin (n x_1)
		+\sum_{n \in \mathbf{H}}(T v)_{n} \cos (n x_1).
	\end{aligned}
\end{equation*}

The higher frequency terms have useful properties presented in Lemma
\ref{Lem:H}.
\begin{lemma}\label{Lem:H}
	If $u(x_1,x_2)$ is the solution to the scattering problem \eqref{eq:u_deter}, then
	$$ \left\|\mathcal{P}_\mathbf{H}\right\|_{A}^{2}
	\simeq \sum_{n \in \mathbf{H}}\left(\partial_{\nu} u\right)_{n} \overline{u_{n}} + \sum_{n \in \mathbf{H}}\left(\partial_{\nu} v\right)_{n}\overline{v_{n}}   \simeq\left\|\mathcal{V}_\mathbf{H}\right\|_{B}^{2} .$$
\end{lemma}
\begin{proof}
	If $n\in \mathbf{H}$,
	\begin{equation*}
		\begin{aligned}
			&\left( \lvert a_{n} \sqrt{k^{2}-n^{2}} \cos \left(\sqrt{k^{2}-n^{2}} b\right) \rvert^2 + \lvert -b_{n} \sqrt{k^{2}-n^{2}} \sin \left(\sqrt{k^{2}-n^{2}} b\right) \rvert^2 \right) \frac{1}{\sqrt{\lvert k^{2}-n^{2}\rvert}+\frac{1}{b}}\\
			\simeq & \left( \lvert a_{n} \sin \left(\sqrt{k^{2}-n^{2}} b\right) \rvert^2 + \lvert  b_{n} \cos \left(\sqrt{k^{2}-n^{2}} b\right) \rvert^2 \right) \left(\sqrt{\lvert k^{2}-n^{2}\rvert}+ \frac{1}{b}\right).
		\end{aligned}
	\end{equation*}
	Hence
	\begin{equation*}
		\left( \lvert  (\partial_{\nu}u)_n \rvert^2 + \lvert  (\partial_{\nu}v)_n \rvert^2 \right) \frac{1}{\sqrt{\lvert k^{2}-n^{2}\rvert}+\frac{1}{b}}
		\simeq \left( \lvert  u_n \rvert^2 + \lvert  v_n \rvert^2 \right) \left(\sqrt{\lvert k^{2}-n^{2}\rvert}+ \frac{1}{b}\right).
	\end{equation*}
	From \eqref{unvn}-\eqref{partialvn}, one can show that for any $n \in \mathbf{H}$,
	$\left(\partial_{\nu} u\right)_{n} \overline{u_{n}} \geq 0$, $\left(\partial_{\nu} v\right)_{n} \overline{v_{n}} \geq 0$.
	Therefore,
	$$ \left\|\mathcal{P}_\mathbf{H} \right\|_{A}^{2}
	\simeq \sum_{n \in \mathbf{H}}\left(\partial_{\nu} u\right)_{n} \overline{u_{n}} + \sum_{n \in \mathbf{H}}\left(\partial_{\nu} v\right)_{n}\overline{v_{n}} .$$
\end{proof}

The next result provides a quantitative estimate of the function $u$ and its normal derivative.

\begin{lemma}\label{Lem:sqrtk}
	If $u(x_1,x_2)$ is the solution to the scattering problem \eqref{eq:u_deter}, then
	\begin{equation}\label{Lemineq}
		\Lambda \sum_{n\in \mathbb{Z} }  \sqrt{\vert {k} ^{2} -\alpha_n^2\vert}  \vert\hat{u} _n\vert^2 \geq \sqrt{\varepsilon}  \sqrt{k}  \|u\|_{L^2(\Gamma)} ^2.
	\end{equation}
\end{lemma}

\begin{proof}
	Since $\Lambda \sum_{n\in \mathbb{Z} }  \vert\hat{u} _n\vert^2 = \|u\|_{L^2(\Gamma)} ^2$ and
	$\sqrt{\vert {k} ^{2} -\alpha_n^2 \vert}  \geq \sqrt{\varepsilon}  \sqrt{k} $ for $\vert k-\alpha_n\vert\geq \varepsilon, \ \forall n \in \mathbb{Z} $,
	$$\Lambda \sum_{n\in \mathbb{Z} }  \sqrt{\vert{k} ^{2} -\alpha_n^2\vert}  \vert\hat{u} _n\vert^2 \geq \sqrt{\varepsilon}  \sqrt{k}  \|u\|_{L^2(\Gamma)} ^2,$$
	which completes the proof.
\end{proof}

\begin{remark}
	Different from Lemma 3.5 in \cite{Bao2012} for the scattering by a cavity, TBC of gratings yield additional difficulties on resonances and  stability.
	Hence, in this paper, a uniform distance $\varepsilon$ is introduced to exclude resonances and thus obtain the estimate \eqref{Lemineq} in Lemma \ref{Lem:sqrtk}, which is essential for the proof procedure of the $k$-explicit stability for periodic structures.
\end{remark}

Next, the relation between the norm of $A$ (or $B$) and the $L^2$ norm is established in Remark \ref{rem:AL}.

\begin{remark}\label{rem:AL}
	For the norm $A$, if there exists a natural number $n^*$ such that $\sqrt{\vert k^2-{n^*}^2 \vert}\leq \frac{1}{b}$, then $n^*$ is very close to $k$ so that $\vert k-{n^*} \vert\leq \frac{1}{ b^2 k}$. Since $\sqrt{\vert k^2-n^2 \vert}\geq \sqrt{\frac{k}{2}}$,
	the norm $A$ can be simplified as
	$$\|u\|_{A} \simeq \sqrt{\left(\sum_{n \in \mathbb{N} \backslash\left\{n^{*}\right\}}\left(\lvert u_{n}\rvert^{2}+\lvert v_{n}\rvert^{2}\right) \frac{1}{\sqrt{\lvert k^{2}-n^{2}\rvert}}\right)+\left(\lvert u_{n^*}\rvert^{2}+\lvert v_{n^*}\rvert^{2}\right) b}.$$	
	If such $n^*$ does not exist, then the norm $A$ has no relation to $b$ since
	$$\|u\|_A \simeq \sqrt{\sum_{n \in \mathbb{N}}\left(\lvert u_{n}\rvert^{2}+\lvert u_{n}\rvert^{2}\right) \frac{1}{\sqrt{\lvert k^{2}-n^{2}\rvert}}}.$$
	Hence, whether $n^*$ exists or not, there always exists a constant $C$ independent of $k$ and $b$ such that
	$$\|u\|_{A}^{2} \leq C\left(\frac{1}{\sqrt{k}}\|u\|_{L^{2}(\Gamma)}^{2}+b \left( \lvert u_{k^{*}}\rvert^{2}+\lvert v_{k^{*}}\rvert^{2}\right)\right),$$
	where $k^*=\arg\min_{n\in\mathbb{N}} \vert n-k\vert $, $u_{k^*} = \frac{1}{\pi} \int_{\Gamma} u(x_1) \sin \left(k^{*} x_1\right) \mathrm{d} x_1 $ and $v_{k^*} = \frac{1}{\pi} \int_{\Gamma} u(x_1) \cos \left(k^{*} x_1\right) \mathrm{d} x_1$.
	It follows immediately that
	\begin{equation}
		\|u\|_{A}^{2} \leq C b\|u\|_{L^{2}(\Gamma)}^{2}.
	\end{equation}
	The similar approximation can be deduced for the norm $B$:
	\begin{equation}\label{ineq:normB}
		b\|u\|_{B}^{2} \leq C \|u\|_{L^{2}(\Gamma)}^{2}.
	\end{equation}
\end{remark}

As for the proof of Theorem \ref{thm:1},
the first step given in Lemma \ref{lem:step1} is to estimate the norm of $u$ in $D$ by the norms of $u$ on $\Gamma$, i.e., the norms $A$ and $B$, and the second step given in Lemma \ref{lem:step2} is to estimate $\|\partial_{\nu}u\|_A$ and $\|u\|_B$ by  $\|g\|_A$. Finally, with the correlation between the norm $A$ and the $L^2$ norm discussed in Remark \ref{rem:AL}, the stability estimate of Theorem \ref{thm:1} is proved.

\begin{lemma}\label{lem:step1}
	Let $u$ be a solution to the scattering problem \eqref{eq:u_deter}. Under Assumption \ref{assump:k}, there exists a constant $C$ independent of $b>1$ and large $k$,  such that
	\begin{equation}
		\|\nabla u \|_{L^2(D)} + k \|u\|_{L^2(D)} \leq C  b k \left( \| \partial_{\nu}u \|_A + \|u\|_B \right),
	\end{equation}
	where the norms $A$ and $B$ are defined in Definition \ref{def:AB}.
\end{lemma}

\begin{lemma}\label{lem:step2}
	Let $u$ be a solution to the scattering problem \eqref{eq:u_deter}. Under Assumption \ref{assump:k}, there exists a constant $C$ independent of $b>1$ and large $k$, such that
	\begin{equation}
		\| \partial_{\nu}u \|_A + \|u\|_B \leq C \max \left\{ \dfrac{  b^{\frac{1} {2} } {k}  }  {\sqrt{\varepsilon} }
		,  b^{\frac{3} {2} }  {k} ^{\frac{3} {2} }  \right\}   \|g\|_A,
	\end{equation}
	where the norms $A$ and $B$ are defined in Definition \ref{def:AB}.
\end{lemma}

\subsection{Proof of Lemma \ref{lem:step1} and Lemma \ref{lem:step2}}

The proof of the bound in Lemma \ref{lem:step1} is mainly based on the definition of the norms and the representation of the series, which means that a stability estimate for the solution $u$ in $D$ can be given in terms of its Dirichlet and Neumann data on $\Gamma$.

\begin{proof}[Proof of Lemma~{\upshape\ref{lem:step1}}]
	From the equation $\Delta u + k^2 u = 0 $ in $D$ and its boundary condition on $S$
	as well as the quasi-periodicity of the solution $u$, one has
	$$ \int_{\Gamma} u \overline{\partial_{\nu} u} \mathrm{d} x_1=\int_{D}\left(\vert\nabla u\vert^{2}-k^{2}\vert u\vert^{2}\right) \mathrm{d} x_1 \mathrm{d} x_2. $$
	
	The left-hand side of the equality above is obviously bounded by $ \left\|\partial_{\nu} u\right\|_{A}^{2}+\|u\|_{B}^{2}$.
	However, the right-hand side has a negative sign for the term $\vert u\vert^2$. In order to prove
	Lemma \ref{lem:step1}, it suffices to show that
	\begin{equation}\label{6.1}
		\|u\|_{L^{2}(D)}^{2} \leq C b^{2}\left(\left\|\partial_{\nu} u\right\|_{A}^{2}+\|u\|_{B}^{2}\right).
	\end{equation}
	
	Note that $\partial_{\nu}u$, $u$ on $\Gamma$, and $u$ in $D$ can be expressed in terms of sine and cosine functions in the direction of $x_1$. Because of the orthogonality of sine and cosine functions, one may assume
	that
	$$ u(x_1,x_2)=\eta \sin \left(\sqrt{k^{2}-n^{2}} x_2\right) \sin (n x_1) + \zeta   \cos \left(\sqrt{k^{2}-n^{2}} x_2\right) \cos (n x_1)\text { in } D, \ n\neq k. $$
	For $n \neq k$, consider three cases:
	(i) $n<k$ and $b \sqrt{k^2-n^2}\geq 1$;
	(ii) $b \sqrt{\vert k^2-n^2 \vert}<1$; and
	(iii) $n>k$ and $b \sqrt{\vert k^2-n^2 \vert}\geq 1$.
	
	For Case (i), it is easy to see that $\|u\|_{L^{2}(D)}^{2} \simeq b$. On the boundary $\Gamma$,
	$$
	\begin{aligned}
		\partial_{\nu} u &=\eta \sqrt{k^{2}-n^{2}} \cos \left(\sqrt{k^{2}-n^{2}} b\right) \sin (n x_1) - \zeta \sqrt{k^{2}-n^{2}} \sin \left(\sqrt{k^{2}-n^{2}} b\right) \cos (n x_1), \\
		u &=\eta \sin \left(\sqrt{k^{2}-n^{2}} b\right) \sin (n x_1) + \zeta   \cos \left(\sqrt{k^{2}-n^{2}} b\right) \cos (n x_1).
	\end{aligned}
	$$
	From the definition of the norms A and B, one has
	$$
	\begin{array}{l}
		\left\|\partial_{\nu} u\right\|_{A}^{2}+\|u\|_{B}^{2} \\
		\qquad \begin{aligned}
			=& \vert\eta\vert^2 \left(\frac{k^{2}-n^{2}}{\sqrt{k^{2}-n^{2}}+\frac{1}{b}}\lvert \cos \left(\sqrt{k^{2}-n^{2}} b\right)\rvert^{2}
			+\left(\sqrt{k^{2}-n^{2}}+\frac{1}{b}\right)\lvert \sin \left(\sqrt{k^{2}-n^{2}} b\right)\rvert^{2} \right)\\
			& + \vert\zeta\vert^2 \left(\frac{k^{2}-n^{2}}{\sqrt{k^{2}-n^{2}}+\frac{1}{b}}\lvert \sin \left(\sqrt{k^{2}-n^{2}} b\right)\rvert^{2}
			+\left(\sqrt{k^{2}-n^{2}}+\frac{1}{b}\right)\lvert \cos \left(\sqrt{k^{2}-n^{2}} b\right)\rvert^{2} \right)\\
			\geq & \frac{1}{2} (\vert\eta\vert^2+\vert\zeta\vert^2)\sqrt{k^{2}-n^{2}} \geq \frac{(\vert\eta\vert^2+\vert\zeta\vert^2)}{2 b} \geq C \frac{1}{b^{2}}\|u\|_{L^{2}(D)}^{2}.
		\end{aligned}
	\end{array}
	$$
	
	For Case (ii), one has
	$$ \lvert \sin \left(\sqrt{k^{2}-n^{2}} x_2\right)\rvert \simeq \sqrt{\lvert k^{2}-n^{2}\rvert} x_2 \text{  and }
	\lvert \cos \left(\sqrt{{k} ^{2} -n^{2} }  x_2\right)\rvert \simeq 1- \dfrac{1} {2}  \lvert {k} ^{2} -n^{2} \rvert x_2^2$$
	from $ \sqrt{\lvert k^{2}-n^{2}\rvert} b<1 $.
	Therefore, $\|u\|_{L^{2}(D)}^{2} \simeq\lvert k^{2}-n^{2}\rvert b^{3}$ and
	$$ \left\|u\left(x, b\right)\right\|_{B}^{2} \simeq\left(\sqrt{\lvert k^{2}-n^{2}\rvert}+\frac{1}{b}\right)\lvert k^{2}-n^{2}\rvert b^{2} \simeq b\lvert k^{2}-n^{2}\rvert, $$
	which implies the inequality \eqref{6.1}.
	
	Finally, for Case (iii),
	$$
	\begin{aligned}
		u&=\eta \sin \left(\sqrt{k^{2}-n^{2}} x_2\right) \sin (n x_1) + \zeta   \cos \left(\sqrt{k^{2}-n^{2}} x_2\right) \cos (n x_1) \\
		&= \eta i \sinh \left(\sqrt{n^{2}-k^{2}} x_2\right) \sin (n x_1) + \zeta  \cosh \left(\sqrt{n^{2}-k^{2}} x_2\right) \cos (n x_1) \text { in } D.
	\end{aligned}
	$$
	The representation of $u$ yields
	$$ \|u\|_{L^{2}(D)}^{2}  \leq C \frac{1}{\sqrt{n^{2}-k^{2}}} e^{2 \sqrt{n^{2}-k^{2}} b} \leq C b^{2}\left(\sqrt{n^{2}-k^{2}}+\frac{1}{b}\right) e^{2 \sqrt{n^{2}-k^{2}} b} \leq C b^{2}\left\|u\left(\cdot, b\right)\right\|_{B}^{2}. $$
\end{proof}

For the proof of Lemma \ref{lem:step2}, the main difficulty is due to the fact that  the boundary data $\partial_{\nu} u = T u +g$ does not yield direct estimates on $\| \partial_{\nu}u \|_A$ and  $\|u\|_B$. To be specific, let us consider the key term $-\mathfrak{Im}\int_{\Gamma} g \bar{u} \mathrm{d} x_1$. Obviously,
\begin{equation}\label{diff}
	- \mathfrak{Im}\int_{\Gamma} g \bar{u} \mathrm{d} x_1 =  \Lambda \sum_{\vert \alpha_n \vert<k}  \sqrt{{k} ^{2} -{\alpha_n} ^2}   \vert\hat{u} _n\vert^2 \geq 0.
\end{equation}
However, the integral in \eqref{diff} does not contain information on $\hat{u}$ for $\vert \alpha_n \vert > k $.
In order to obtain the stability, we have to analyze $\hat{u}$ for $\vert \alpha_n \vert > k $, and thus to generate a bound on $u$. Next, we give the proof of Lemma \ref{lem:step2}.

\begin{proof}[Proof of Lemma~{\upshape\ref{lem:step2}}]
	From the definition,
	\begin{equation*}
		\begin{aligned}
			\int_{\Gamma} g \bar{u} \mathrm{d} x_1 & = \int_{\Gamma} (\partial_{\nu} u - T(u)) \bar{u} \mathrm{d} x_1\\
			& = \int_{\Gamma}  \partial_{\textbf{n} }  u \bar{u}  \mathrm{d} x_1- i  \Lambda \sum_{\vert n \vert<k}  \sqrt{{k} ^{2} -n^2}   \vert\hat{u} _n\vert^2  +
			\Lambda \sum_{\vert n\vert>k}  \sqrt{n^2-{k} ^{2} }  \vert \hat{u} _n \vert^2 \\
			& \triangleq \mathfrak{Re} + i \mathfrak{Im},
		\end{aligned}
	\end{equation*}
	with
	$\mathfrak{Re}  = \int_{\Gamma}  \partial_{\textbf{n} }  u \bar{u}  \mathrm{d} x_1 + \Lambda \sum_{\vert n \vert>k}  \sqrt{n^2-{k} ^{2} }  \vert \hat{u} _n \vert^2$ and
	$ \mathfrak{Im}   = - \Lambda \sum_{\vert n \vert<k}  \sqrt{{k} ^{2} -n^2}   \vert \hat{u} _n \vert^2  $.
	Obviously,
	\begin{equation}\label{case2Im}
		\Lambda \sum_{\vert n \vert<k}  \sqrt{{k} ^{2} -n^2}   \vert \hat{u} _n \vert^2 = \lvert \mathfrak{Im}\rvert \leq \lvert  \int_{\Gamma} g \bar{u}  \mathrm{d} x_1 \rvert.
	\end{equation}
	As for the real part, since
	$$
	\lvert\mathfrak{Re}\rvert \leq \lvert  \int_{\Gamma} g \bar{u} \mathrm{d} x_1 \rvert \leq \pi\|g\|_A \|u\|_B \leq \gamma \pi \left\| \partial_{\nu}u  \right\|_A \|u\|_B
	$$
	provieded that 
	\begin{equation}\label{condcase2}
		\| g \|_A \leq \gamma \| \partial_{\nu}u  \|_A,
	\end{equation}
	and
	\begin{equation*}
		\begin{aligned}
		\mathfrak{Re} = &\sum_{n \in \mathbf{L}} \left( \partial_{\nu}u \right)_n \bar{u}_n + \sum_{n \in \mathbf{L}} \left( \partial_{\nu}v \right)_n \bar{v}_n + \sum_{n \in \mathbf{H}} \left( \partial_{\nu}u \right)_n  \bar{u}_n + \sum_{n \in \mathbf{H}} \left( \partial_{\nu}v \right)_n \bar{v}_n \\
		&+\Lambda \sum_{\vert n \vert>k}  \sqrt{n^2-{k} ^{2} }  \vert \hat{u} _n \vert^2,
		\end{aligned}
	\end{equation*}
	one obtains
	\begin{equation}\label{case2Re}
		\begin{aligned}
			&\gamma \pi \left\|  \partial_{\nu}u  \right\|_A \|u\|_B +  \lvert   \sum_{n \in \mathbf{L}} \left( \partial_{\nu}u \right)_n \bar{u}_n + \sum_{n \in \mathbf{L}} \left( \partial_{\nu}v \right)_n \bar{v}_n \rvert \\
			& \geq \sum_{n \in \mathbf{H}}\left(\partial_{\nu} u\right)_{n} \overline{u_{n}} + \sum_{n \in \mathbf{H}}\left(\partial_{\nu} v\right)_{n}\overline{v_{n}}  
			 + \Lambda \sum_{\vert n \vert>k}  \sqrt{n^2-{k} ^{2} }  \vert \hat{u} _n \vert^2,
		\end{aligned}
	\end{equation}
	with
	$$ \sum_{n \in \mathbf{H}}\left(\partial_{\nu} u\right)_{n} \overline{u_{n}} + \sum_{n \in \mathbf{H}}\left(\partial_{\nu} v\right)_{n}\overline{v_{n}} 
	 \simeq \left\| \mathcal{P}_\mathbf{H} \right\|_A^2 \geq 0$$
	and
	$$\Lambda \sum_{\vert n \vert>k}  \sqrt{n^2-{k} ^{2} }  \vert \hat{u} _n \vert^2 \geq 0.
	$$

	The estimate \eqref{case2Re} implies that  $\displaystyle\sum_{n \in \mathbf{H}}\left(\partial_{\nu} u\right)_{n} \overline{u_{n}} + \sum_{n \in \mathbf{H}}\left(\partial_{\nu} v\right)_{n}\overline{v_{n}}  $ and $ \Lambda \displaystyle\sum_{\vert n \vert>k}  \sqrt{n^2-{k} ^{2} }  \vert \hat{u} _n \vert^2$ can be bounded by $\lvert   \displaystyle\sum_{n \in \mathbf{L}} \left( \partial_{\nu}u \right)_n \bar{u}_n + \sum_{n \in \mathbf{L}} \left( \partial_{\nu}v \right)_n \bar{v}_n \rvert$ 
since $\gamma \pi \left\| \left( \partial_{\nu}u \right)_n  \right\|_A \|u\|_B$ may be absorbed by choosing an appropriate $\gamma$.

Evidently,  in order to estimate  $ \Lambda \displaystyle\sum_{\vert n \vert>k}  \sqrt{n^2-{k} ^{2} }  \vert \hat{u} _n \vert^2$, it is sufficient to estimate $
	\lvert   \displaystyle\sum_{n \in \mathbf{L}} \left( \partial_{\nu}u \right)_n \bar{u}_n + \sum_{n \in \mathbf{L}} \left( \partial_{\nu}v \right)_n \bar{v}_n \rvert\;.
	$
	Since the estimate \eqref{case2Re} is based on the condition \eqref{condcase2}, we will consider $\| g \|_A \leq \gamma \| \partial_{\nu}u  \|_A$ and $\| g \|_A \geq \gamma \| \partial_{\nu}u  \|_A$ separately. For the case $\| g \|_A \leq \gamma \| \partial_{\nu}u  \|_A$, since the term $
	\lvert   \displaystyle\sum_{n \in \mathbf{L}} \left( \partial_{\nu}u \right)_n \bar{u}_n + \sum_{n \in \mathbf{L}} \left( \partial_{\nu}v \right)_n \bar{v}_n \rvert
	$ concerns the low frequency part for ${u}_n$ and $\left( \partial_{\nu}u \right)_n$, to characterize which term is bigger, $\delta \left\| \mathcal{P}_\mathbf{L}  \right\|_A \geq  \left\|  \mathcal{V}_\mathbf{L}   \right\|_B$ and $\delta \left\| \mathcal{P}_\mathbf{L}  \right\|_A \leq  \left\|  \mathcal{V}_\mathbf{L}   \right\|_B$ are considered separately. 	
	To sum up, we consider the following three cases:
	\begin{itemize}
		\item Case (i)
		\begin{equation}\label{case2.1}
			\| g \|_A \leq \gamma \| \partial_{\nu}u  \|_A, \quad \delta \left\| \mathcal{P}_\mathbf{L}  \right\|_A \geq  \left\|  \mathcal{V}_\mathbf{L}   \right\|_B;
		\end{equation}
		\item Case (ii)
		\begin{equation}\label{case2.2}
			\| g \|_A \leq \gamma \| \partial_{\nu}u  \|_A, \quad \delta \left\| \mathcal{P}_\mathbf{L}  \right\|_A \leq  \left\|  \mathcal{V}_\mathbf{L}   \right\|_B.
		\end{equation}
	    \item Case (iii)
	    \begin{equation}\label{case1.1}
	    	\| g \|_A \geq \gamma \| \partial_{\nu}u  \|_A;
	    \end{equation}
	\end{itemize}
	The constant $\gamma$ and $\delta$ are small numbers which will be chosen later.

	For Case (i),
	it yields from \eqref{case2Re} that
	\begin{equation}\label{case2.1.1}
		\begin{aligned}
		&\delta \left\| \mathcal{P}_\mathbf{L}  \right\|_A^2  + \gamma \pi \left\| \left(  \partial_{\nu}u \right)_n  \right\|_A \|u\|_B  \\
		&\geq  \left(\sum_{n \in \mathbf{H}} \left( \partial_{\nu}u \right)_n  \bar{u}_n + \sum_{n \in \mathbf{H}}\left(\partial_{\nu} v\right)_{n}\overline{v_{n}}  \right)
		+\Lambda \sum_{\vert n \vert>k}  \sqrt{n^2-{k} ^{2} }  \vert \hat{u} _n \vert^2.
		\end{aligned}
	\end{equation}
	For the second term in \eqref{case2.1.1}, Case (i) and Lemma \ref{Lem:H} yield
	\begin{equation}\label{case2.1.2}
		\begin{aligned}
			\gamma \pi \left\| \left( \partial_{\nu}u \right)_n  \right\|_A \|u\|_B
			\leq & \gamma \pi \left( \left\| \mathcal{P}_\mathbf{L}  \right\|_A + C \left\| \mathcal{V}_\mathbf{H} \right\|_B \right) \times \left( \delta \left\| \mathcal{P}_\mathbf{L}  \right\|_A + C \left\| \mathcal{V}_\mathbf{H}  \right\|_B \right)\\
			\leq & C' \left( \gamma \delta \left\| \mathcal{P}_\mathbf{L}  \right\|_A^2 +  \gamma^2 \left\| \mathcal{P}_\mathbf{L}  \right\|_A^2 + \frac{1}{2} \left\| \mathcal{V}_\mathbf{H}  \right\|_B^2 \right).
		\end{aligned}
	\end{equation}
	Since
	\begin{equation}\label{case2.1frac12}
		\gamma^2 \| \partial_{\nu}u  \|_A^2
		\geq  \| \partial_{\nu}u -T(u) \|_A^2
		\geq  \left\| \mathcal{P}_\mathbf{L}  - \mathcal{T}_\mathbf{L}   \right\|_A^2
		\geq   \frac{1}{2} \left\| \mathcal{P}_\mathbf{L} \right\|_A^2 - \left\| \mathcal{T}_\mathbf{L} \right\|_A^2,
	\end{equation}
	one has
	\begin{equation}\label{case2.1.3}
		\frac{1}{2} \left\| \mathcal{P}_\mathbf{L} \right\|_A^2  \leq  \left\| \mathcal{T}_\mathbf{L} \right\|_A^2 + \gamma^2 \| \partial_{\nu}u  \|_A^2
		\leq \left\| \mathcal{T}_\mathbf{L} \right\|_A^2  + \gamma^2 \left( \left\| \mathcal{P}_\mathbf{L} \right\|_A^2 + C \left\| \mathcal{V}_\mathbf{H} \right\|_B^2 \right).
	\end{equation}
	By \eqref{case2.1.1}, \eqref{case2.1.2} and \eqref{case2.1.3},
	choosing
	\begin{equation}\label{gamma}
		\gamma=C_{(i)}\sqrt{\delta},
	\end{equation}
	where $C_{(i)}$ is an  arbitrarily sufficiently small constant independent of $k$ and $b$, it can be deduced  that
	\begin{equation}\label{Tu_L}
		\left\| \mathcal{T}_\mathbf{L} \right\|_A^2 \geq \dfrac{C_1}{\delta} \left( \left\| \mathcal{V}_\mathbf{H} \right\|_B^2 + \Lambda \sum_{\vert n \vert>k}  \sqrt{n^2-{k} ^{2} }  \vert \hat{u} _n \vert^2 \right).
	\end{equation}
	
	In order to get a bound for $\left\| \mathcal{T}_\mathbf{L} \right\|_A^2$, it is sufficient to estimate $(Tu)_n$ and $(Tv)_n$ for $n\leq k$ due to the definition of $\mathbf{L}$.
	By the Cauchy-Schwarz inequality, one has
	\begin{equation}
		\begin{aligned}
			&  \sum_{\vert n \vert<k}  \sqrt{{k} ^{2} -n^2}   \vert \hat{u} _n \vert^2  +   \sum_{\vert n \vert>k}  \sqrt{n^2-{k} ^{2} }  \vert \hat{u} _n \vert^2 \\
			\geq & \lvert  \sum_{n\in \mathbb{Z} }  \sqrt{{k} ^{2} -n^2}   \hat{u} _n \overline{\widehat{\sin(n x_1)} } \rvert^2 \left( \sum_{n\in \mathbb{Z} }  \sqrt{\vert{k} ^{2} -n^2\vert}   \lvert \widehat{\sin(n x_1)} \rvert^2   \right)^{-1}  \\
			\geq & \lvert (Tu)_n \rvert^2 \left( \sum_{n\in \mathbb{Z} }  \sqrt{\vert{k} ^{2} -n^2\vert}   \lvert \widehat{\sin(n x_1)} \rvert^2 \right)^{-1}  \\
			\geq & \frac{1} {2k}  \lvert (Tu)_n \rvert^2 \quad n < k,
		\end{aligned}
	\end{equation}
	and
	\begin{equation}  \label{CSv}
		\begin{aligned}
			& \sum_{\vert n \vert<k}  \sqrt{{k} ^{2} -n^2}   \vert \hat{u} _n \vert^2  +   \sum_{\vert n \vert>k}  \sqrt{n^2-{k} ^{2} }  \vert \hat{u} _n \vert^2 \\
			\geq & \lvert  \sum_{n\in \mathbb{Z} }  \sqrt{{k} ^{2} -n^2} \hat{u} _n \overline{\widehat{\cos(n x_1)} }   \rvert^2 \left( \sum_{n\in \mathbb{Z} }  \sqrt{\vert{k} ^{2} -n^2\vert}  \lvert \widehat{\cos(n x_1)} \rvert^2   \right)^{-1}  \\
			\geq & \lvert (Tv)_n \rvert^2 \left( \sum_{n\in \mathbb{Z} }  \sqrt{\vert{k} ^{2} -n^2\vert}  \lvert \widehat{\cos(n x_1)} \rvert^2 \right)^{-1}  \\
			\geq & \frac{1} {2k}  \lvert (Tv)_n \rvert^2 \quad n < k.
		\end{aligned}
	\end{equation}
	
	Note that
	\begin{equation}\label{LCy0}
	\begin{aligned}
		\sum_{n\in\mathbf{L}} \frac{1}{\sqrt{\lvert k^{2}-n^{2}\rvert}+\frac{1}{b}}  &= \sum_{n=0}^{\lfloor k \rfloor-1}\frac{1}{\sqrt{\lvert k^{2}-n^{2}\rvert}+\frac{1}{b}} +  \frac{1}{\sqrt{\lvert k^{2}-\lfloor k \rfloor^2  \rvert}+\frac{1}{b}} \\
		&\leq \displaystyle\int_0^{\lfloor k \rfloor} \frac{1}{\sqrt{\lvert k^{2}-x^{2}\rvert}+\frac{1}{b}} \mbox{d}x + b \\
		&\leq C + b \leq {C_2}b
	\end{aligned}	
	\end{equation}
	since $b>1$.
	A combination of the above estimate \eqref{LCy0} with \eqref{Tu_L}-\eqref{CSv} yields
	\begin{equation} \label{case2.1leqgeq}
		\begin{aligned}
			& C_2  b k \left(  \Lambda \sum_{\vert n \vert<k}  \sqrt{{k} ^{2} -n^2}   \vert \hat{u} _n \vert^2  +  \Lambda \sum_{\vert n \vert>k}  \sqrt{n^2-{k} ^{2} }  \vert \hat{u} _n \vert^2 \right) \\
			\geq & \sum_{n\in\mathbf{L} }  \frac{1} {\sqrt{\lvert {k} ^{2} -n^{2} \rvert} +\frac{1} {b} }  \left( \lvert (Tu)_n \rvert^2 + \lvert (Tv)_n \rvert^2  \right) \\
			= & \left\| \mathcal{T}_\mathbf{L}  \right\|_A^2 \\
			\geq & \frac{C_1} {\delta}  \left( \left\|\mathcal{V}_\mathbf{H}\right\|_B^2 +  \Lambda \sum_{\vert n \vert>k}  \sqrt{n^2-{k} ^{2} }  \vert \hat{u} _n \vert^2 \right) \\
			\geq & \frac{C_1} {\delta}  \Lambda \sum_{\vert n \vert>k}  \sqrt{n^2-{k} ^{2} }  \vert \hat{u} _n \vert^2.
		\end{aligned}
	\end{equation}
	By choosing
	\begin{equation}\label{delta}
		\delta = \frac{C_{(ii)}}{ b k},
	\end{equation}
	with $C_{(ii)} = \frac{2 C_1}{1 + 2 C_2}>0$, which is independent of $k$ and $\varepsilon$, one has
	\begin{equation}\label{ineq:Cnk}
		C  \sum_{\vert n \vert<k}  \sqrt{{k} ^{2} -n^2}   \vert \hat{u} _n \vert^2  \geq \sum_{\vert n \vert>k}  \sqrt{n^2-{k} ^{2} }  \vert \hat{u} _n \vert^2.
	\end{equation}
	Applying the above estimate \eqref{ineq:Cnk}  and the choice of $\delta$ \eqref{delta} to \eqref{case2.1leqgeq}, one has
	\begin{equation}\label{ineq:bk}
		C  \sum_{\vert n \vert<k}  \sqrt{{k} ^{2} -n^2}   \vert \hat{u} _n \vert^2 \geq  \left\|\mathcal{V}_\mathbf{H}\right\|_B^2 +  \Lambda \sum_{\vert n \vert>k}  \sqrt{n^2-{k} ^{2} }  \vert \hat{u} _n \vert^2.
	\end{equation}
	By combining \eqref{case2Im}, \eqref{ineq:bk}, Lemma \ref{Lem:sqrtk} and the proposition \eqref{ineq:normB} of norm $B$, one gets
	\begin{equation*}
		\begin{aligned}
			C  k \lvert  \int_{\Gamma}  g \bar{u}  \mathrm{d} x_1\rvert
			\geq &  k \left\| \mathcal{V}_\mathbf{H} \right\|_B^2 + C'' \sqrt{\varepsilon}   k \sqrt{k}  \|u\|_{L^2(\Gamma)} ^2
			\geq  k \left\| \mathcal{V}_\mathbf{H} \right\|_B^2 + \sqrt{\varepsilon}   \sqrt{k}   \left\| \mathcal{V}_\mathbf{L} \right\|_B^2
			\geq  \sqrt{\varepsilon}   \sqrt{k}   \|u\|_B^2.
		\end{aligned}
	\end{equation*}
	Hence,
	\begin{equation}\label{case21uB}
		\dfrac{ C \sqrt{k} } {\sqrt{\varepsilon} }  \|g\|_A \geq \|u\|_B.
	\end{equation}
	
	Moreover, since Lemma \ref{Lem:H} indicates
	\begin{equation}\label{case2.1.H}
		\left\| \mathcal{V}_\mathbf{H} \right\|_B^2 \geq C \left\| \mathcal{P}_\mathbf{H} \right\|_A^2,
	\end{equation}
	and \eqref{case2.1frac12} implies
	\begin{equation}\label{case2.1frac12.2}
		\left\| \mathcal{T}_\mathbf{L} \right\|_A^2 \geq \frac{1}{2} \left\| \mathcal{P}_\mathbf{L} \right\|_A^2 - \gamma^2 \| \partial_{\nu}u  \|_A^2,
	\end{equation}
	a combination of the above \eqref{case2.1.H}-\eqref{case2.1frac12.2} with \eqref{case2Im},  \eqref{case2.1leqgeq} and Lemma \ref{Lem:H} yields
	\begin{equation*}
		\begin{aligned}
			C b k \pi \|g\|_A \|u\|_B
			\geq  C b k \lvert \int_{\Gamma} g \bar{u} \mathrm{d} x_1\rvert
			\geq  C b k \Lambda \sum_{\vert n \vert<k}  \sqrt{{k} ^{2} -n^2}   \vert \hat{u} _n \vert^2
			\geq   \sqrt{\varepsilon} \left\| \partial_{\nu}u \right\|_A^2.
		\end{aligned}
	\end{equation*}
	Therefore, using \eqref{case21uB}, one has
	\begin{equation*}
		\dfrac{C b {k} ^{\frac{3} {2} }     } {\varepsilon } \|g\|_A^2 \geq \left\| \partial_{\nu}u \right\|_A^2.
	\end{equation*}
	Hence,
	\begin{equation}\label{case21puA}
		\dfrac{C b^{\frac{1} {2} } {k} ^{\frac{3} {4} }     } {\sqrt{\varepsilon} } \|g\|_A \geq \left\| \partial_{\nu}u \right\|_A.
	\end{equation}
	The estimate \eqref{case21uB} and \eqref{case21puA} indicate that
	\begin{equation}\label{case21result}
		\dfrac{C b^{\frac{1} {2} } {k} ^{\frac{3} {4} }     } {\sqrt{\varepsilon} } \|g\|_A \geq \left\| \partial_{\nu}u \right\|_A + \|u\|_B.
	\end{equation}
	
	For Case (ii),
	Applying Lemma \ref{Lem:H} to \eqref{case2Re}, one has
	\begin{equation}\label{case2.2.0.1}
		\frac{1}{\delta} \left\| \mathcal{V}_\mathbf{L} \right\|_B^2 + \gamma\pi \| \partial_{\nu}u \|_A \|u\|_B  \geq C \left\| \mathcal{V}_\mathbf{H} \right\|_B^2 + \Lambda \sum_{\vert n \vert>k}  \sqrt{n^2-{k} ^{2} }  \vert \hat{u} _n \vert^2.
	\end{equation}
	For the second term in \eqref{case2.2.0.1}, Case (ii) and Lemma \ref{Lem:H} yield
	\begin{equation}\label{case2.2.0.2}
		\begin{aligned}
			\gamma\pi \|\partial_{\nu}u \|_A \|u\|_B
			\leq & \gamma\pi \left( \frac{1}{\delta} \left\| \mathcal{V}_\mathbf{L} \right\|_B + C \left\| \mathcal{V}_\mathbf{H} \right\|_B \right)  \times \left( \left\| \mathcal{V}_\mathbf{L} \right\|_B + \left\| \mathcal{V}_\mathbf{H} \right\|_B \right) \\
			\leq & C \frac{\gamma}{\delta} \left\| \mathcal{V}_\mathbf{L} \right\|_B^2 + C \gamma^2 \left\| \mathcal{V}_\mathbf{L} \right\|_B^2 + \frac{1}{2}\left\| \mathcal{V}_\mathbf{H} \right\|_B^2.
		\end{aligned}
	\end{equation}
	Since $\gamma=C_{(i)}\sqrt{\delta}$ is sufficiently small, the estimate \eqref{case2.2.0.1} and \eqref{case2.2.0.2} indicate that
	\begin{equation}\label{case2.2.0}
		\frac{C}{\delta} \left\| \mathcal{V}_\mathbf{L} \right\|_B^2
		\geq \left\| \mathcal{V}_\mathbf{H} \right\|_B^2 + \Lambda \sum_{\vert n \vert>k}  \sqrt{n^2-{k} ^{2} }  \vert \hat{u} _n \vert^2  \geq \left\| \mathcal{V}_\mathbf{H} \right\|_B^2.
	\end{equation}
	On the other hand, for the lower frequency, it is obvious that
	\begin{equation*}
		C   \Lambda \sum_{\vert n \vert<k}  \sqrt{{k} ^{2} -n^2}   \vert \hat{u} _n \vert^2  \geq \left\| \mathcal{V}_\mathbf{L} \right\|_B^2.
	\end{equation*}
	Hence, it follows from \eqref{case2Im} that
	\begin{equation*}
		C  \lvert \int_{\Gamma}  g \bar{u}  \mathrm{d} x_1\rvert  \geq C   \Lambda \sum_{\vert n \vert<k}  \sqrt{{k} ^{2} -n^2}   \vert \hat{u} _n \vert^2   \geq \left\| \mathcal{V}_\mathbf{L} \right\|_B^2.
	\end{equation*}
	Thus
	\begin{equation}\label{case2.2.1}
		C   \|g\|_A \|u\|_B \geq \left\| \mathcal{V}_\mathbf{L} \right\|_B^2.
	\end{equation}
	By \eqref{case2.2.0} and $\delta=\dfrac{C_{(ii)}}{b k}$, one has
	\begin{equation*}
		C  b k \|g\|_A \|u\|_B \geq \left\| u \right\|_B^2.
	\end{equation*}
	Thus
	\begin{equation}\label{case2.2.2}
		C  b k \|g\|_A \geq \left\| u \right\|_B.
	\end{equation}
	
	Note that
	\begin{equation}\label{case2.2.3}
		\|u\|_B \geq \left\| \mathcal{V}_\mathbf{L} \right\|_B \geq \dfrac{C_{(ii)}}{ bk} \left\| \mathcal{P}_\mathbf{L}  \right\|_A,
	\end{equation}
	due to the assumption of Case (ii).
	Applying \eqref{case2.2.3} to \eqref{case2.2.1}, one gets
	\begin{equation*}
		C  b^2 {k} ^2\|g\|_A \|u\|_B \geq \left\| \mathcal{P}_\mathbf{L}  \right\|_A^2.
	\end{equation*}
	Using \eqref{case2.2.0} and Lemma \ref{Lem:H}, this implies that
	\begin{equation}\label{case2.2.4}
		C b^2  {k} ^2 \|g\|_A \|u\|_B \geq \|\partial_{\nu}u \|_A^2.
	\end{equation}
	Combining the above estimate \eqref{case2.2.4} with \eqref{case2.2.2} yields
	\begin{equation}\label{case2.2.5}
		C   b^{\frac{3} {2} } {k} ^{\frac{3} {2} }  \|g\|_A \geq \|\partial_{\nu}u \|_A.
	\end{equation}
	The estimate \eqref{case2.2.2} and \eqref{case2.2.5} indicate that
	\begin{equation}\label{case22result}
		C b^{\frac{3} {2} } {k} ^{\frac{3} {2} }    \|g\|_A \geq \|\partial_{\nu}u \|_A + \|u\|_B.
	\end{equation}

	For Case (iii), it suffices to estimate $\|u\|_B$.
	Using the definition of $T(u)$, the property of the norm defined above and Lemma \ref{Lem:sqrtk},
	one obtains
	\begin{equation} \label{ineq1}
		\begin{aligned}
			\left( 1+\frac{1} {\gamma}  \right)  \sqrt{k}  \pi \|g\|_A \|u\|_B
			& \geq \sqrt{k}  \pi \|T(u)\|_A \|u\|_B \\
			& \geq  \sqrt{k}  \lvert  		\Lambda \sum_{n\in \mathbb{Z} }  \sqrt{{k} ^{2} -n^2}  \vert \hat{u} _n \vert^2 		\rvert
			\left( \geq \sqrt{k}  \lvert
			\frac{\Lambda} {2}  \sum_{n\in \mathbb{Z} }  \sqrt{\vert {k} ^{2} -n^2\vert}  \vert \hat{u} _n \vert^2
			\rvert \right) \\
			& \geq \dfrac{\sqrt{\varepsilon} } {2}  k\|u\|_{L^2(\Gamma)} ^2.
		\end{aligned}
	\end{equation}
	Since $k \geq \sqrt{k^2-n^2}$  for $n\in \mathbf{L}$, it is obvious that there exists a constant $C$ such that
	\begin{equation}\label{kinL}
		Ck \geq \sqrt{\lvert k^{2}-n^{2}\rvert}+ \frac{1}{b}.
	\end{equation}
	Applying  \eqref{kinL} to \eqref{ineq1}, the bound for $n\in\mathbf{L}$ can be deduced:
	\begin{equation}\label{uL}
		C \dfrac{\sqrt{k}  } {\sqrt{\varepsilon}  \gamma}  \|g\|_A \|u\|_B
		\geq \sum_{n \in \mathbf{L}} \left(\sqrt{\lvert k^{2}-n^{2}\rvert}+ \frac{1}{b}\right) \left( \vert u_n \vert^2 + \vert v_n \vert^2 \right) = \left\| \mathcal{V}_\mathbf{L} \right\|_B^2.
	\end{equation}
	Using Lemma \ref{Lem:H}, the bound for $n\in \mathbf{H}$ can be deduced:
	\begin{equation}\label{uH}
		\frac{1}{\gamma^2} \|g\|_A^2 \geq \| \partial_{\nu}u  \|_A^2 \geq \left\| \mathcal{P}_\mathbf{H} \right\|_A^2 \simeq \left\| \mathcal{V}_\mathbf{H} \right\|_B^2.
	\end{equation}
	Combining \eqref{uL} and \eqref{uH} yields
	\begin{equation*}
		C \left( \frac{1}{\gamma^2} \|g\|_A^2+\dfrac{\sqrt{k}  } {\sqrt{\varepsilon}  \gamma}\|g\|_A \|u\|_B \right) \geq \|u\|_B^2.
	\end{equation*}
	Solving the second order polynomial above, one has
	\begin{equation}\label{case1.2}
		C  \dfrac{\sqrt{k}  } {\sqrt{\varepsilon}  \gamma}  \|g\|_A \geq \|u\|_B.
	\end{equation}
	Therefore, for Case (i), one obtains from \eqref{case1.1} and \eqref{case1.2} that
	\begin{equation}\label{case1result}
		C  \dfrac{\sqrt{k}  } {\sqrt{\varepsilon} \gamma} \|g\|_A \geq \| \partial_{\nu}u  \|_A + \|u\|_B.
	\end{equation}	
	Applying choice of $\gamma$ \eqref{gamma} and $\delta$ \eqref{delta} to the estimate \eqref{case1result} immediately yields that
	\begin{equation} \label{case1resultnew}
		C  \dfrac{ b^{\frac{1} {2} } k } {\sqrt{\varepsilon} }  \|g\|_A \geq \| \partial_{\textbf{n} } u  \|_A + \|u\|_B.
	\end{equation}

	Finally, we complete the proof of Lemma \ref{lem:step2} based on the above proofs for Case (i)-(iii).
	Combining the bounds  \eqref{case21result} , \eqref{case22result} and \eqref{case1resultnew} for Case (i)-(iii) respectively, one arrives
	\begin{equation*}
		\| \partial_{\textbf{n} } u \|_A + \|u\|_B \leq C \max \left\{ \dfrac{ b^{\frac{1} {2} } k  } {\sqrt{\varepsilon} }
		, b^{\frac{3} {2} } k ^{\frac{3} {2} }   \right\}   \|g\|_A,
	\end{equation*}
	which completes the proof.
\end{proof}

\subsection{Proof of Theorem \ref{thm:1}}
\label{sec:proofthm1}

Now we are ready to prove the bounds in Theorem \ref{thm:1}.

\begin{proof}
	Lemma \ref{lem:step2} yields that: Under Assumption \ref{assump:k},
	there exists a constant $C$ independent of $b>1$ and large $k$, such that
	\begin{equation*}
		\| \partial_{\nu}u \|_A + \|u\|_B \leq C \max \left\{ \dfrac{  b^{\frac{1} {2} } k } {\sqrt{\varepsilon} }
		,  b^{\frac{3} {2} } {k} ^{\frac{3} {2} }    \right\}  \|g\|_A .
	\end{equation*}
	Furthermore, Lemma \ref{lem:step1} gives
	\begin{equation*}
		\|\nabla u \|_{L^2(D)} + k \|u\|_{L^2(D)}  \leq C  b k \left( \| \partial_{\nu}u \|_A + \|u\|_B \right) .
	\end{equation*}
	Therefore,
	\begin{equation}\label{eq:2.1}
		\| \partial_{\nu}u \|_A + \|u\|_B \leq C  \max \left\{ \dfrac{ b^{\frac{3} {2} } {k} ^2  } {\sqrt{\varepsilon} }
		,  b^{\frac{5} {2} }  {k} ^{\frac{5} {2} }   \right\} \|g\|_A .
	\end{equation}
	From the definition and the property of the norm $A$ mentioned in Definition \ref{def:AB} and Remark \ref{rem:AL},
	\begin{equation}\label{eq:2.2}
		\|g\|_A \leq C \sqrt{b} \|g\|_{L^2(\Gamma)} .
	\end{equation}
	
	Applying \eqref{eq:2.2} to \eqref{eq:2.1}, one derives the bounds in Theorem \ref{thm:1}, which completes the proof.
\end{proof}

\section{Stability for the random case}
\label{sec:thm23}

In this section, Theorem \ref{thm:3} is proved by  utilizing the general framework developed in  \cite{Pembery2020} based upon the stability in Theorem \ref{thm:1} for the deterministic case.
Note that the randomness of integral domain prevents direct application of general framework in  \cite{Pembery2020}.
Thus the difficulty lies in that the scattering surface in model problem is
with randomness, which implies the stochastic domain for the scattering problem.
For this, a variable transform, which changes the random variational form with stochastic domains into a transformed one with a definite domain and random medium, is introduced to prove all necessary propositions, including the continuity of  sesquilinear form $\tilde{a}$ and antilinear functional $\tilde{G}$, and stability and uniqueness which hold almost surely, which implies the stability for the random case.


\subsection{Prerequisites}
\label{sec:4.preliminary}

Denote $\tilde{a}_{c(\omega)}$ and $\tilde{G}_{c(\omega)}$ (defined by \eqref{aw} and \eqref{Gw} in Sec \ref{sec:variational}) by $(\tilde{a} \circ c)(\omega)$ and $(\tilde{G} \circ c)(\omega)$, respectively.
Define the norm $\|v\|^2_{1,k} := \|\nabla v \|_{L^2(D)}^2 + k^2 \|v\|_{L^2(D)}^2$ and $\|v\|_{1,\infty}:=\|v\|_{\infty}+\|v'\|_{\infty}$ on $H_{S,qp}^1(D)$.
Necessary propositions required in the general framework \cite{Pembery2020} include the continuity of $\tilde{a}$ and $\tilde{G}$, regularity of $\tilde{a}\circ c$ and $\tilde{G}\circ c$, measurability and $\mu$-essentially separability of $c$, stability a.s. and uniqueness a.s., described in Proposition \ref{prop:1} and Proposition \ref{prop:2}.

\begin{proposition}
	\label{prop:1}
	$\tilde{a}$, $\tilde{G}$ and $c$ in the variational form \eqref{variational} of the scattering problem in a random periodic structure satisfies the following properties:
	\begin{itemize}
		\item[(i)] The function $c: \Omega \rightarrow \mathcal{C}$ defined by \eqref{def:c} is measurable and $\mu$-essentially separably valued.
		\item[(ii)]
		Let $B(H_{S,qp}^1(D), H_{S,qp}^1(D))$ denote the space of bounded linear maps $H_{S,qp}^1(D) \rightarrow H_{S,qp}^1(D)$. The functions $\tilde{a}: \mathcal{C}\rightarrow B(H_{S,qp}^1(D),H_{S,qp}^1(D))$ and $\tilde{G}:\mathcal{C}\rightarrow H_{S,qp}^1(D)$ defined by \eqref{aw} and \eqref{Gw} are continuous,
		the maps $\tilde{a}\circ c\in L^{\infty}(\Omega;B(H_{S,qp}^1(D),H_{S,qp}^1(D)))$, $\tilde{G}\circ c\in L^2(\Omega;H_{S,qp}^1(D))$.
	\end{itemize}
\end{proposition}

\begin{definition}
	(The solution operator $\mathcal{U}$)
	Define $\mathcal{U}:\mathcal{C}\rightarrow H_{S,qp}^1(D)$ by letting $\mathcal{U}(f_0) \in H_{S,qp}^1(D)$ be the solution of the deterministic Helmholtz problem with sampling $f_0\in\mathcal{C}$.
\end{definition}

\begin{proposition} \label{prop:2}
	For $f_0\in Lip$,  the solution of the variational form \eqref{variational} exists and is unique.
	Let $v_0=\mathcal{U}(f_0)$ be the solution. Under Assumption \ref{assump:k}, there exists a constant $C$ independent of  $b>1$ and large $k$, such that
	\begin{equation}\label{Destimatelem}
		\|\nabla v_0 \|_{L^2(D)} + k \|v_0\|_{L^2(D)} \leq C \max \left\{ \dfrac{b^2 {k} ^{2} } {\sqrt{\varepsilon} } ,b^3 {k} ^{\frac{5} {2} }   \right\}  \|g\|_{L^2(\Gamma)}.
	\end{equation}
	Moreover, the scattering problem \eqref{eq:u_total} in a random periodic structure satisfies the following stability and uniqueness properties:
	\begin{itemize}
		\item[(i)] (Uniqueness almost surely) $ker(\tilde{a}_{c(\omega)})=\{0\}$ $\mu$-almost surely.
		\item[(ii)] (Stability almost surely) There exist $C_1, h_1:\Omega\rightarrow\mathbb{R}$ such that $C_1 h_1\in L^1(\Omega)$ and the bound
		\begin{equation}
			\|u(\omega)\|_{1,k}^2\leq C_1(\omega) h_1(\omega)
		\end{equation}
		holds almost surely.
	\end{itemize}
\end{proposition}

\begin{remark}
	From the general framework in  \cite{Pembery2020}, one can conclude that (i) with properties of $\tilde{a}$, $\tilde{G}$ and $c$ stated in Proposition \ref{prop:1}, the maps $\mathcal{A}$ and $\mathcal{G}$ defined by \eqref{def:AA} and \eqref{def:GG} are well-defined;
	(ii) with the almost surely uniqueness property in Proposition \ref{prop:2}, the solution to the Helmholtz equation in the stochastic case is unique in $L^2(\Omega;H_{S,qp}^1(D))$;
	(iii) with the almost surely stability property in Proposition \ref{prop:2}, integrability and measurability implies the stochastic a priori bound.
\end{remark}

In order to verify Proposition \ref{prop:1} and Proposition \ref{prop:2}, we give the following four lemmas. Lemma \ref{lem:Kirsch} means that the solution to the scattering problem depends continuously on the scattering surface.
Lemma \ref{lem:Pettis} shows the equivalence between strong measurability and measurability as well as $\mu$-essentially separability.
Lemma \ref{lem:Bochner} shows the necessary and sufficient condition for Bochner integrability.
Lemma \ref{u_welldefined} shows that the solution operator $\mathcal{U}$ is well defined.

\begin{lemma}\label{lem:Kirsch}
	(continuous dependence on the boundary curve $f$   \cite{Kirsch1993}) Let $f\in C^2(\mathbb{R})$ be $2 \pi$-periodic and $u\in H_{S,qp}^1$ be the unique solution of the scattering problem \eqref{eq:u_deter}. Let $K \subset D$ be compact. Then there exists $\gamma>0$ and $C>0$ both depending on $k, u^i, f$ and $K$, such that for all $2\pi$-periodic $r\in C^1(\mathbb{R})$ with $\|r-f\|_{1,\infty}\leq\gamma$, the unique solution $u_r$ of the scattering problem corresponding to $r$ satisfies
	$$ \|u_r-u\|_{H^1(K)} \leq C \|f-r\|_{1,\infty}.$$
	Here, $\|q\|_{1,\infty}:=\|q\|_{\infty}+\|q'\|_{\infty}$ denotes the norm in $C^1[0,\Lambda]$, where $\Lambda$ is the period of the scattered structures.
\end{lemma}

\begin{lemma}\label{lem:Pettis}
	(Pettis measurability theorem
	(Proposition 2.15 in  \cite{Ryan2002})). Let  $(\Omega, \mathcal{F}, \mu)$ be a complete $\sigma$-finite measure space. The following are equivalent for a function $f: \Omega \rightarrow X$ (i) $f$ is strongly measurable, (ii) $f$ is measurable and $\mu$-essentially separably valued.
\end{lemma}

\begin{lemma}\label{lem:Bochner}
	(Bochner's Theorem (Proposition 2.16 in \cite{Ryan2002})). If $f: \Omega \rightarrow X$ is a strongly measurable function, then $f$ is Bochner integrable if and only if the scalar function $\|f\|$ is integrable.
\end{lemma}

\begin{lemma}\label{u_welldefined}
	( $\mathcal{U}$ is well defined)
	For $f_0\in\mathcal{C}$, the solution  $\mathcal{U}$ of the scattering problem exists, is unique, and depends continuously on $f_0$.
\end{lemma}

\begin{proof}	
	For $f_0\in\mathcal{C}$, the scattering problem is a deterministic case satisfying the condition in Theorem \ref{thm:1}. Based on the previous work in   \cite{Bao1995} on the existence and uniqueness, it follows from Theorem \ref{thm:1} and Lemma \ref{lem:Kirsch} that the solution $\mathcal{U}$ is  well defined.
\end{proof}

Here the basic propositions on measure theory and Bochner spaces needed for the proof are omitted. See  \cite{Bogachev2007} and  \cite{Diestel1977} for more details.

\subsection{Proof of Proposition \ref{prop:1} and Proposition \ref{prop:2}}\label{sec:4.verification}

In this section, we verify that the random periodic structure scattering problem has all necessary propositions required as in the general framework \cite{Pembery2020}.
Proposition \ref{prop:1} includes
measurability  and
$\mu$-essentially separability of $c$, continuity of $\tilde{a}  $ and $\tilde{G}  $ and
regularity of $\tilde{a}  \circ c$ and  $\tilde{G}  \circ c$.
Proposition \ref{prop:2} includes the necessary
stability and uniqueness which hold almost surely. To be specific, under the assumpiton of resonance exclusion and stochastic regularity of the scattering surface, a variable transform is introduced and theory of Bochner spaces such as Pettis measurability theorm and Bochner's theorem is used to complete the proof.

First, we give the proof of Proposition \ref{prop:1}.

\begin{proof}[Proof of Proposition~{\upshape\ref{prop:1}}]
	Since each of $f$ is a Lipschitz function, $f$ is strongly measurable. By Pettis measurability theorem (Lemma \ref{lem:Pettis}), it follows that $c$ defined by Definition \ref{def:c} is measurable and $\mu$-essentially separably valued, so properties of $c$ in Proposition \ref{prop:1} are satisfied.
	
	In order to prove the continuity of $\tilde{a}$ and $\tilde{G}$, we need to show that if $(f_m)\rightarrow(f_0)$ in $\mathcal{C}$ then $\tilde{a}(f_m)\rightarrow \tilde{a}(f_0)$ in $B(H_{S,qp}^1(D),H_{S,qp}^1(D))$, and similarly for $\tilde{G}$. Since
	\begin{equation}\label{varia:f_m}
		\tilde{a}_{f_m}(v_m,\phi)=\tilde{G}_{f_m}(\phi),
	\end{equation}
	where
	\begin{equation}
		\tilde{a}_{f_m} =\int_{D_{f_m}}  \nabla v_m \cdot \nabla \bar{\phi}
		- k^2 \int_{D_{f_m}} v_m \bar{\phi}
		- \int_{\Gamma} T \tau v_m \overline{\tau \phi}  , \quad  \tilde{G}_{f_m} = \int_{\Gamma} g \overline{\tau \phi} ,
	\end{equation}
	and
	\begin{equation}\label{varia:f_0}
		\tilde{a}_{f_0}(v_0,\phi)=\tilde{G}_{f_0}(\phi),
	\end{equation}
	where
	\begin{equation}\label{varia:f_0aG}
		\tilde{a}_{f_0} =\int_{D_{f_0}}  \nabla v_0 \cdot \nabla \bar{\phi}
		- k^2 \int_{D_{f_0}} v_0 \bar{\phi}
		- \int_{\Gamma} T \tau v_0 \overline{\tau \phi} , \quad  \tilde{G}_{f_0} = \int_{\Gamma} g \overline{\tau \phi}.
	\end{equation}
	Our goal is to  transform the volume integral in $D_{f_m}$ into an integral over $D_{f_0}$ by a suitable change of the variable $x$.
	
	Choose $\gamma_0>0$ such that
	\begin{equation}\label{def:gamma}
		\gamma_0<\min\{\min\{x_2:(x_1,x_2)\in D_{f_0}\}-f_0(x_1):x_1\in\mathbb{R}\},
	\end{equation} and a function $\alpha\in C^1(\mathbb{R})$ with
	\begin{equation}\label{def:alpha}
		\alpha(t) =\left\{ \begin{array}{ll}
			1, & t \leq \gamma_0/2 \\
			0, & t\geq \gamma_0.
		\end{array} \right..
	\end{equation}	
	For $f_m$ such that $\|f_m-f_0\|_{1,\infty} \leq \gamma_0$, define $\mathcal{F}: D_{f_0} \rightarrow D_{f_m}$ by
	$$
	\mathcal{F}(y) = y_2 + \alpha(y_2-f_0(y_1)) [f_m(y_1)-f_0(y_1)] \hat{e}_2,\quad y\in D_{f_0},
	$$
	where $\hat{e}_2=(0,1)^T$.	
	For $\|f_m-f_0\|_{1,\infty} \leq \gamma$ with sufficiently small $\gamma\leq \gamma_0$, $\mathcal{F}$ is a diffeomorphism and $\mathcal{F}=\mathcal{I}$ on $K$. The Jacobian $J_{\mathcal{F}}$ of $\mathcal{F}$ satisfies $J_{\mathcal{F}}(y) = \mathcal{I} + \mathcal{O}(\|f_m-f_0\|_{1,\infty})$ uniformly in $y\in D_{f_0}$. For the inverse $\mathcal{Q}$ of $\mathcal{F}$, one also has $J_{\mathcal{Q}}(y) = \mathcal{I} + \mathcal{O}(\|f_m-f_0\|_{1,\infty})$. The change of variable $x=\mathcal{F}(y)$ then transforms \eqref{varia:f_m} into
	\begin{equation}\label{eq:varia_new}
		\int_{D_{f_0}}\left[\sum_{i, j=1}^{2} b_{i j} \frac{\partial \hat{v}_{m}}{\partial y_{i}} \frac{\partial \overline{\hat{\phi}}}{\partial y_{j}}-k^{2} \hat{v}_{m} \overline{\hat{\phi}}\right] \operatorname{det} (J_{\mathcal{F}}) \mathrm{d} y-\int_{\Gamma} \overline{\tau \hat{\phi}} T \tau \hat{v}_{m} \mathrm{d} s=\int_{\Gamma} g \overline{\tau \hat{\phi}} \mathrm{d} s,
	\end{equation}
	where $\hat{v}_{m}=v_{m} \circ \mathcal{F}, \hat{\phi}=\phi \circ \mathcal{F}$ both in $H_{S,qp}^1(D_{f_0})$, and
	\begin{equation}\label{bij}
		b_{i j}(y)=\sum_{l=1}^{2} \frac{\partial \mathcal{Q}_{i}(x)}{\partial x_{l}} \frac{\partial \mathcal{Q}_{j}(x)}{\partial x_{l}}\rvert_{x=\mathcal{F}(y)}, \quad i, j = 1, 2.
	\end{equation}
	The left hand side of \eqref{eq:varia_new} defines a sesquilinear form $\tilde{a}_{f_m}(v,\phi)$ on $H_{S,qp}^1(D_{f_0})$.
	Since $ \operatorname{det} (J_{\mathcal{F}}) = 1 + \mathcal{O}(\|f_m-f_0\|_{1,\infty}) $ and $ b_{ij}=\delta_{ij}+\mathcal{O}(\|f_m-f_0\|_{1,\infty}) $, one concludes that
	$$\lvert \tilde{a}_{f_m}(v, \phi)-\tilde{a}_{f_0}(v, \phi)\rvert \leq C\|f_m-f_0\|_{1, \infty}\|v\|_{H^{1}\left(D_{f_0}\right)}\|\phi\|_{H^{1}\left(D_{f_0}\right)} \text { for all } v, \phi \in H_{S,qp}^1(D_{f_0}).$$
	Hence if $(f_m)\rightarrow(f_0)$ in $\mathcal{C}$, then $\tilde{a}_{f_m}\rightarrow \tilde{a}_{f_0}$ in $B(H_{S,qp}^1(D),H_{S,qp}^1(D))$. For $\tilde{G}$, it follows from the definition \eqref{varia:f_0}-\eqref{varia:f_0aG}, the right hand side of \eqref{eq:varia_new},  $\hat{\phi}=\phi \circ \mathcal{F} $ and $\phi=\phi \circ \mathcal{I}$ that
	if $(f_m)\rightarrow(f_0)$ in $\mathcal{C}$, then $\tilde{G}_{f_m}\rightarrow \tilde{G}_{f_0}$ in $H_{S,qp}^1(D)$.
	
	Since $c$ is strongly measurable and the map $\tilde{a}$ is continuous,  $\tilde{a}\circ c $ is strongly measurable. Recall that the operator $T$ is continuous from $H^{1/2}(\Gamma)$ to $H^{-1/2}(\Gamma)$   \cite{Nedelec2001}. For $v\in H_{S,qp}^1(D)$, $\phi\in H_{S,qp}^1(D)$, one has
	\begin{equation}\label{eq:afomega}
		\tilde{a}_{f(\omega)}(v, \phi) =\int_{D_{f(\omega)}}  \nabla v \cdot \nabla \bar{\phi}
		- k^2 \int_{D_{f(\omega)}} v \bar{\phi}
		- \int_{\Gamma} T \tau v \overline {\tau \phi}.
	\end{equation}
	Same as above, transform the volume integral in $D_{f(\omega)}$ into an integral over $D_{f_0}$.
	For $f_1=f(\omega)$ such that $\|f_1-f_0\|_{1,\infty} \leq \gamma_0$ (defined same as \eqref{def:gamma}), define $\mathcal{F}: D_{f_0} \rightarrow D_{f_1}$ by
	$$
	\mathcal{F}(y) = y_2 + \alpha(y_2-f_0(y_1)) [f_1(y_1)-f_0(y_1)] \hat{e}_2,\quad y\in D_{f_0},
	$$
	where $\hat{e}_2=(0,1)^T$ and $\alpha$ is defined by \eqref{def:alpha}.	
	For $\|f_1-f_0\|_{1,\infty} \leq \gamma$ with sufficiently small $\gamma \leq \gamma_0$, $\mathcal{F}$ is a diffeomorphism and $\mathcal{F}=\mathcal{I}$ on $K$.  The change of variable $x=\mathcal{F}(y)$ then transforms \eqref{eq:afomega} into
	\begin{equation}
		\tilde{a}_{f(\omega)}(\hat{v}, \phi)=\int_{D}\left[\nabla \hat{v}\operatorname{det} (J_{\mathcal{F}}) B \nabla \bar{\phi}-k^{2}(\operatorname{det} (J_{\mathcal{F}}) \hat{v} \bar{\phi}\right] d y,
	\end{equation}
	where $\hat{v}=v \circ \mathcal{F}, \hat{\phi}=\phi \circ \mathcal{F}$ both in $H_{S,qp}^1(D_{f_0})$, and $B=(b_{ij})_{i,j=1,2}$ from \eqref{bij}. 		
	Since $ \operatorname{det} (J_{\mathcal{F}}) = 1 + \mathcal{O}(\|f_m-f_0\|_{1,\infty}) $ and $ b_{ij}=\delta_{ij}+\mathcal{O}(\|f_m-f_0\|_{1,\infty}) $, observe that the Cauchy-Schwarz inequality and properties of $T$ imply that there exists $C>0$ such that
	$$
	\lvert \tilde{a}_{f(\omega)}(\hat{v}, \phi)\rvert \leq
	C \|f(\omega)-f_0\|_{1,\infty}^2 \|v\|_{H^1(D)} \|\phi\|_{H^1(D)} \text { for all } v, \phi \in H_{S,qp}^1(D_{f_0}),
	$$
	and hence $\tilde{a}\circ c \in L^{\infty}(\Omega;B(H_{S,qp}^1(D),H_{S,qp}^1(D))) $.
	
	Since $c$ is strongly measurable and the map $\tilde{G}$ is continuous,  $\tilde{G}\circ c $ is strongly measurable. It is clear that $\| (\tilde{G}\circ c)(\omega) \|_{H_{S,qp}^1(D)} \leq \|g\|_{L^2(D)}$, and thus $\tilde{G}\circ c \in L^2(\Omega;H_{S,qp}^1(D))$ since $g \in L^2(\Omega;L^2(D))$.
\end{proof}

Now prove Proposition \ref{prop:2}.

\begin{proof}[Proof of Proposition~{\upshape\ref{prop:2}}]
	For $f_0\in \mathcal{C}$, it follows from Theorem \ref{thm:1} and Lemma \ref{lem:Kirsch} that the solution $\mathcal{U}(f_0)$ of the variational problem exists, is unique, and has the $k$-explicit stability estimate \eqref{Destimatelem}. 	
	Uniqueness almost surely holds immediately from Lemma \ref{u_welldefined}.
	
	For stability almost surely, choose
	$C_1 =C'  \max \left\{ \dfrac{b^2 {k} ^{2} } {\sqrt{\varepsilon} } ,b^3 {k} ^{\frac{5} {2} }   \right\} $ with $C'=\sup_\Omega {C}$, $C$ being the previous constant in \eqref{Destimatelem}, and $h_1 = \|g\|_{L^2(\Gamma)}$.
	It remains to show that $C_1 h_1 \in L^1(\Omega)$.
	We first show that $C_1 h_1$ is measurable and then show that it lies in $ L^1(\Omega)$.
	
	To conclude $C_1$ is measurable, use the fact that the product, sum and maximum of two measurable functions are measurable. Since $g$ is measurable and the map $g\mapsto \|g\|_{L^2(\Gamma)}^2$ is clearly continuous, $h_1$ is measurable. As the product of two measurable functions is measurable, it follows that $C_1 h_1$ is measurable.
	
	Now show that $C_1 h_1 \in L^1(\Omega)$. Using the Cauchy-Schwarz inequality yields
	\begin{eqnarray} \label{C1h1L1}
		\left\|C_{1}h_1\right\|_{L^{1}(\Omega)}=\int_{\Omega} C_{1}(\omega)h_1 (\omega) \mathrm{~d} \mathbb{P}(\omega)\leq \left\|C_{1}\right\|_{L^{1}(\Omega)}\left\|\left\|g\right\|_{L^{2}\left(\Gamma\right)}^{2}\right\|_{L^{1}(\Omega)}<\infty.
	\end{eqnarray}
	Therefore $C_1 h_1 \in L^1(\Omega)$ as required. Integrating \eqref{Destimatelem} in the probability space and using \eqref{C1h1L1} obviously yield \eqref{Sestimate}.
\end{proof}

\subsection{Proof of Theorem \ref{thm:3}}
\label{sec:4.proof}

Before the proof, the continuity of the solution operator $\mathcal{U}$ is given in the following lemma.

\begin{lemma}\label{u_continuity}
	(Continuity of $\mathcal{U}$)
	For the scattering problem by a random periodic structure, the solution operator $\mathcal{U}:\mathcal{C}\rightarrow H_{S,qp}^1(D)$ is continuous.
\end{lemma}
\begin{proof}
	Let $f_0, f_1 \in \mathcal{C}$, with $\mathcal{U}(f_0)=u_0$ and $\mathcal{U}(f_1)=u_1$. Then for any $v\in H_{S,qp}^1(D)$,
	$$\tilde{a}_{f_j}(u_j,v)=\tilde{G}_{f_j}(v),\ j=0,1.$$
	
	Since one has
	$$\lvert \tilde{a}_{f_1}(u, \phi)-\tilde{a}_{f_0}(u, \phi)\rvert \leq C\|f_1-f_0\|_{1, \infty}\|u\|_{H^{1}\left(D_{f_0}\right)}\|\phi\|_{H^{1}\left(D_{f_0}\right)} \text { for all } u, \phi \in H_{S,qp}^1(D_{f_0})$$
	from the proof of Proposition \ref{prop:1},
	where
	both $\tilde{a}_{f_1}(u, \phi)$
	and $\tilde{a}_{f_0}(u, \phi)$ are sesquilinear forms,
	one can applies the general perturbation theory of variational equation \cite{Kato1976} which yields
	$$\|u_1\circ\mathcal{F} - u_0\|_{H^1(D_{f_0})} = \|\hat{u}_1-u_0\|_{H^1(D_{f_0})}\leq C \|f_m-f_0\|_{1,\infty},$$
	and thus, since $\mathcal{F}=\mathcal{I}$ on $D_{f_0}$,
	$$\|u_1-u_0\|_{H^1(K)} \leq C \|f_1-f_0\|_{1,\infty}.$$
	Let $u_d:=u_0-u_1$. It can be deduced that  $u_d\rightarrow 0$ in $H_{S,qp}^1(D_{f_0})$ as $f_1 \rightarrow f_0$ in $\mathcal{C}$. This ends the proof of this lemma.
\end{proof}

Now we are ready to prove Theorem \ref{thm:3}.

\begin{proof}
	Let $u = \mathcal{U} \circ c$ (which
	is well-defined by Lemma \ref{u_welldefined}). By construction, $a_{c(\omega)}(u(\omega), v)=G_{c(\omega)}(v)$ for all $v \in H_{S,qp}^1(D)$ almost surely.
	Since it follows from Assumption \ref{cond:f} and Lemma \ref{u_continuity} that $u$ is measurable, $u$ solves the variational problem.
	
	Under Assumption \ref{cond:f}, continuity of $\tilde{a}  $ and $\tilde{G}  $,
	regularity of $\tilde{a}  \circ c$ and  $\tilde{G}  \circ c$, and
	measurability  and
	$\mu$-essentially separability of $c$ hold by Proposition \ref{prop:1};
	stability a.s. and uniqueness a.s. hold by Proposition \ref{prop:2}.
	Moreover, since there is no trapping cases for the scattering problem by a Lipschitz perfectly conductor, the nontrapping condition required in   \cite{Pembery2020} is naturally satisfied.
	Therefore
	the maps $\mathcal{A}$ and $\mathcal{G}$ (defined by \eqref{def:AA} and \eqref{def:GG}) are well-defined;
	there exists $u\in L^2(\Omega;H^1_{S,qp}(D))$ being the solution to Problem \ref{P3}; the solution to Problem \ref{P3} is unique in $L^2(\Omega;H^1_{S,qp}(D))$.

	Combining the stability for the scattering problem by deterministic periodic structures in Theorem \ref{thm:1} with integrability and measurability properties on stochastic quantities verified by Proposition \ref{prop:1} and stability and uniqueness which hold almost surely given by Proposition \ref{prop:2} yields the well-posedness and the $k$-explicit stability for the scattering problem by random periodic structures, which completes the proof of Theorem \ref{thm:3}.
\end{proof}

\begin{remark}
	If the random structure is (uncertainty) quantified using the Karhunen-Lo\`eve (KL) expansion as in   \cite{BaoLinSINUM2020}, then the random structure can be represented by the following  Karhunen-Lo\`eve expansion
	\begin{equation*}
		\begin{aligned}
			f(\omega;x_1)&=\tilde{f}(x_1) + \displaystyle \sum_{j=0}^{\infty}\sqrt{\lambda_j} \xi_{j}(\omega)\varphi_{j}(x_1),
		\end{aligned}
	\end{equation*}
	where $\tilde{f}(x_1)$ is a $\Lambda$-periodic deterministic function, eigenvalues $\{ \lambda_j \}_{j=0}^{\infty}$ arranged in a descending order are corresponding to the orthonormal eigenfunctions $\{ \varphi_{j} \}_{j=0}^{\infty}$ of the covariance operator $C_f$
	and $\{\xi_j \}_{j=0}^{\infty}$ is a random variable with zero mean and unit covariance. The covariance function   \cite{bookRandomSurface, bookcovariance} takes the following form
	\begin{equation*}
		c(x_1-y_1)=\sigma^2 \exp(-\dfrac{\vert x_1-y_1\vert^2}{l^2}), 0<l\ll \Lambda,
	\end{equation*}
	where $\sigma$ is the root mean square of the surface and $l$ is the correlation length.
	It follows that the KL expansion of the random process $f(\omega;x_1)$ may be written as
	\begin{equation*}\label{eq:fKL}
		\begin{aligned}
			f(\omega;x_1)&=\tilde{f}(x_1)+\sqrt{\lambda_0}\xi_0(\omega) \sqrt{\dfrac{1}{\Lambda}} \\
			& + \displaystyle \sum_{j=1}^{\infty}\sqrt{\lambda_j}\left( \xi_{j,s}(\omega)\sqrt{\dfrac{2}{\Lambda}} \sin\left(\dfrac{2j\pi x_1}{\Lambda}\right)
			+ \xi_{j,c}(\omega)\sqrt{\dfrac{2}{\Lambda}} \cos\left(\dfrac{2j\pi x_1}{\Lambda}\right) \right),
		\end{aligned}
	\end{equation*}
	where $\xi_0$, $\xi_{j,s}$ and $\xi_{j,c}$ are independent and identically distributed random variables with zero mean and unit covariance.
	
	Since the eigenvalues $\{ \lambda_j \}_{j=0}^{\infty}$ decrease exponentially, it is easy to show that the derivative of the random surface $f(\omega;x_1)$ respect to $x_1$ is bounded and there exists a positive constant $C$ such that $\vert f(\omega;x_1)-f(\omega;x_2)\vert \leq C \vert x_1-x_2\vert$. Thus the random surface $f(\omega;x_1)$ in KL expansion is Lipschitz. Therefore, we have the results on the well-posedness of variational formulations of the Helmholtz equation on a random periodic structure proved in this paper, which verifies the well-posedness of the direct problem in   \cite{BaoLinSINUM2020}.
\end{remark}

\section{Conclusions}
\label{sec:conslusion}

In this work, we have established stability results for the Helmholtz equation on deterministic and random periodic structures, respectively.  Both estimates are explicit with respect to the wavenumber. To the authors' best knowledge, these are the first stability results explicit with respect to the wavenumber for the scattering problem in periodic structures.
An interesting future direction is to apply the stability results for conducting convergence analysis of numerical methods for solving the scattering problems. 
Our techniques may be applicable to other (stochastic) scattering problems. In particular, another future direction is to
conduct stability analysis for the electromagnetic scattering problems.

%
%
%

\bmhead{Acknowledgments}

This work was supported in part by National Natural Science
Foundation of China (11621101, U21A20425, 12071430, 12201404), a Key Laboratory of Zhejiang Province, and Postdoctoral Science Foundation of China (2021TQ0203).

\section*{Declarations}


\begin{itemize}
	\item[] \textbf{Conflict of interest} The authors declare that they have no conflict of interest.
	\item[] \textbf{Data availability} Data sharing is not applicable to this article as obviously no datasets were generated or analyzed during the current study.
	\item[] \textbf{Publisher's Note} Springer Nature remains neutral with regard to jurisdictional claims in
	published maps and institutional affiliations.
\end{itemize}






\end{document}